\documentclass[a4paper, 11pt, twoside]{article}

\usepackage[a4paper, margin=0.6in]{geometry}
\usepackage[utf8]{inputenc}
\usepackage[T1]{fontenc}
\usepackage[english]{babel}
\usepackage{textcomp} 
\usepackage{graphicx}
\DeclareGraphicsExtensions{.jpg,.mps,.pdf,.png,.gif}
\usepackage[french]{varioref} 
\usepackage[]{amsmath,amssymb,amsfonts}

\usepackage{float}
\usepackage{bbm}
\usepackage{amsthm}
\usepackage{dsfont}
\usepackage{color}
\usepackage{natbib}

\usepackage[hyperindex,breaklinks]{hyperref}

\newcommand{\Ben}{\begin{enumerate}}
\newcommand{\Een}{\end{enumerate}}
\newcommand{\Bit}{\begin{itemize}}
\newcommand{\Eit}{\end{itemize}}
\newcommand{\Beq}{\begin{equation}}
\newcommand{\Eeq}{\end{equation}}
\newcommand{\Ba}{\begin{align*}}
\newcommand{\Ea}{\end{align*}}
\newcommand{\Mb}{\mathbf}

\newtheorem{Th}{Theorem}
\newtheorem{Lem}{Lemma}
\newtheorem{Prop}{Proposition}

\newtheorem{Rq}{Remark} 
\newtheorem{Corr}{Corollary}

\newcommand{\Tb}{\textcolor{black}}

\begin{document}
\title{\Tb{Extremal dependence and spatial risk measures for insured losses due to extreme winds}}
\date{December 18, 2019}

\author{
~~ Erwan Koch\footnote{EPFL, Institute of Mathematics, EPFL SB MATH MATH-GE, MA B1 457 (B\^atiment MA), Station 8, 1015 Lausanne, Switzerland. ETH Zurich (Department of Mathematics, RiskLab). \newline Email: erwan.koch@epfl.ch}}

\maketitle

\begin{abstract}
\Tb{
A meticulous assessment of the risk of impacts associated with extreme wind events is of great necessity for populations, civil authorities as well as the insurance industry. Using the concept of spatial risk measure and related set of axioms introduced by \cite{koch2017spatial, koch2019SpatialRiskAxioms}, we quantify the risk of losses due to extreme wind speeds. The insured cost due to wind events is proportional to the wind speed at a power ranging typically between $2$ and $12$. Hence we first perform a detailed study of the correlation structure of powers of the \Tb{Brown--Resnick max-stable random fields} and look at the influence of the power.} Then, using the latter results, we thoroughly investigate spatial risk measures associated with variance and induced by powers of max-stable random fields. In addition, we show that spatial risk measures associated with several classical risk measures and induced by such cost fields satisfy (at least part of) the previously mentioned axioms \Tb{under conditions which are generally satisfied for the risk of damaging extreme wind speeds. In particular, we specify the rates of spatial diversification in different cases, which is valuable for the insurance industry.} 

\medskip

\noindent \textbf{Key words:} \Tb{Extreme wind speed}; \Tb{Insurance}; Powers of max-stable random fields; \Tb{Reinsurance}; Spatial diversification; Spatial risk measures and corresponding axioms; Wind damage.  

\medskip

\noindent \textbf{Mathematics Subject Classification (2000):} 60F05; 60G60; 60G70; 91B30.
\end{abstract}


\section{Introduction}
\label{Sec_Intro}

\Tb{Extratropical cyclones (such as European windstorms) and tropical cyclones constitute a major risk for society, as can be seen from the consequences in Europe of Windstorms Lothar and Martin in December 1999 (140 fatalities and damage of around 19 billion USD) and, in many Caribbean islands and parts of Florida, of Hurricane Irma in September 2017 (at least 134 deaths and damage exceeding 67.8 billion USD). An accurate evaluation of this risk is essential for civil authorities and the insurance industry.}

Having in mind the spatial nature of environmental extreme events, \cite{koch2017spatial, koch2019SpatialRiskAxioms} introduced a new notion of spatial risk measure, which makes explicit the contribution of space and \Tb{allows one} to take into account the spatial dependence of the cost field in the risk measurement. The spatial risk measure associated with a classical risk measure $\Pi$ and induced by a cost random field $C$ (e.g., modelling the cost due to damage caused by a windstorm) is the function of space resulting from the application of $\Pi$ to the normalized integral of $C$ on various geographical areas. \cite{koch2017spatial, koch2019SpatialRiskAxioms} also proposed a set of axioms characterising how the value of an induced spatial risk measure is expected to evolve with respect to the space variable, at least under some conditions on $\Pi$ and $C$. To the best of our knowledge, the papers by \cite{koch2017spatial, koch2019SpatialRiskAxioms} are the first articles establishing a theory about risk measures in a spatial context where the risks spread over a geographical region. This theory is of interest for the insurance industry as it allows, for instance, \Tb{the quantification of} the rate of spatial diversification. 

One of the main goals of the present paper is to \Tb{apply this notion of spatial risk measure and related axioms to analyze the risk of losses due to extreme wind speeds. We model extreme wind speeds using max-stable random fields \citep[e.g.,][]{haan1984spectral, de2007extreme, davison2012statistical}, which are especially suitable to model the temporal maxima of a given variable at all points in space since they constitute the only possible non-degenerate limiting field of pointwise maxima taken over suitably rescaled independent copies of a random field \citep[e.g.,][]{haan1984spectral}. Moreover, we consider the power-law damage function $D(z)=(z/c_1)^{\beta}$, $z>0$, for $\beta \in \mathbb{N} \backslash \{ 0 \}$ and $c_1>0$}, which is particularly adapted \Tb{to wind hazard. Our cost field model arises from the application of $D$ to max-stable fields. Although powers $\beta$ equalling $2$ or $3$ are justified by physical arguments (wind load and dissipation rate of wind kinetic energy, respectively) for the effective cost, it has been shown that they may be much higher when insured costs are considered. For instance, \cite{prahl2012applying} find exponents ranging from $8$ to $12$ for residential buildings in Germany and argue that such large exponents stem for instance from the presence of a deductible in the insurance contract. Hence an important aspect of this paper will be to study how the spatial dependence and risk evolve with respect to that power.} 

First, we thoroughly investigate the correlation structure of powers of the Brown--Resnick max-stable random field. This part contains new theoretical results for the \Tb{Brown--Resnick field} and, therefore, may be of interest \Tb{for the extreme-value community independently of any risk-related consideration. Moreover, we perform a numerical study using typical values for the parameters of the generalized extreme-value (GEV) distribution governing annual wind speed maxima, and show that the correlation between damage at two stations basically does not depend on the value of the damage power. This is useful news for insurance companies owing to the large range for possible values of that power.} Then, we study several spatial risk measures induced by powers of some max-stable random fields, \Tb{mainly the Brown--Resnick fields}. Using the first part, we theoretically study spatial risk measures associated with variance and induced by such cost fields. \Tb{This analysis is supplemented with a numerical study where, again, we look at the influence of the value of the power $\beta$}. Moreover, we show that, \Tb{under conditions which are generally satisfied for the risk of losses due to extreme wind}, spatial risk measures associated with several classical risk measures and induced by powers of some max-stable fields satisfy (at least part of) the axioms introduced in \cite{koch2017spatial, koch2019SpatialRiskAxioms}. \Tb{Inter alia, we know the rates of spatial diversification in the cases where the classical risk measures are the variance, the value-at-risk (VaR) and the expected shortfall (ES)}. The obtained results \Tb{may} have useful implications for the insurance industry \Tb{and, throughout the study, we keep a strong connection with concrete actuarial practice.}

The remainder of the paper is organized as follows. In Section \ref{Sec_Reminder_Spatial_Risk_Measures}, we briefly \Tb{expose} the notion of spatial risk measure and the corresponding set of axioms introduced in \cite{koch2017spatial, koch2019SpatialRiskAxioms}. Moreover, we specify the cost field model underlying the examples of spatial risk measures considered\Tb{; we especially review the literature about wind damage functions and provide a short introduction to max-stable fields.} Section \ref{Sec_Dependence_Power_Maxstable} investigates the correlation structure of powers of the \Tb{Brown--Resnick} random fields. In Section \ref{Sec_Examples_SRM_Power_Damage_Function}, we study some spatial risk measures induced by \Tb{those} cost fields. Finally, Section \ref{Sec_Conclusion} provides a short summary as well as some perspectives. Throughout the paper, the elements belonging to $\mathbb{R}^d$ for some $d \geq 1$ are denoted \Tb{by} bold symbols, whereas those in more general spaces are designated using normal font. Moreover, $\nu$ stands for the Lebesgue measure in $\mathbb{R}^d$ and $\mathbb{N}_*=\mathbb{N} \backslash \{ 0 \}$. Finally, $\overset{d}{=}$ and $\overset{d}{\rightarrow}$ denote equality and convergence in distribution, respectively. In the case of random fields, distribution has to be understood as the set of all finite-dimensional distributions. 

\section{Spatial risk measures and cost field model}
\label{Sec_Reminder_Spatial_Risk_Measures}

\subsection{Spatial risk measures and corresponding axioms}
\label{Subsec_Spat_Risk_Axioms}

Let $\mathcal{A}$ be the set of all compact subsets of $\mathbb{R}^2$ with a positive Lebesgue measure and $\mathcal{A}_c$ the set of all convex elements of $\mathcal{A}$. Denote by $\mathcal{C}$ the set of all real-valued and measurable\footnote{\Tb{Throughout the paper, when applied to random fields, the adjective ``measurable'' means ``jointly measurable''.}} random fields on $\mathbb{R}^2$ having almost surely (a.s.) locally integrable sample paths. Each field characterizes the economic or insured cost generated by the events belonging to specified categories and occurring during a given time period, say $[0,T_L]$. In the following, $T_L$ is considered as fixed and does not appear for the sake of notational simplicity. Each category of events (e.g., \Tb{European windstorms or hurricanes}) will be named a hazard in the following. Let $\mathcal{L}$ be the set of all real-valued random variables defined on an adequate probability space. A risk measure is some function $\Pi: \mathcal{L} \to \mathbb{R}$ and will be referred to as a classical risk measure throughout the paper in order to \Tb{avoid confusion} with a spatial risk measure. A classical risk measure $\Pi$ is said to be law-invariant if, for all $X \in \mathcal{L}$, $\Pi (X)$ only depends on the distribution of $X$.

\Tb{\citet{koch2017spatial, koch2019SpatialRiskAxioms} introduced the normalized spatially aggregated loss function, defined as
\Beq
\label{Eq_NormSpatAggreLoss}
L_N(A,C)= \dfrac{1}{\nu(A)} \displaystyle \int_{A} C(\Mb{x}) \  \nu(\mathrm{d}\Mb{x}), \quad A \in \mathcal{A}, C \in \mathcal{C},
\Eeq
which disentangles the contribution of the space and the contribution of the hazards.}
The quantity $L_N(A,C)$ represents the economic or insured loss per surface unit on region $A$ due to hazards whose costs can be modelled with the cost field $C$. 
\Tb{A spatial risk measure $\mathcal{R}_{\Pi}$ as introduced in \cite{koch2017spatial, koch2019SpatialRiskAxioms} is defined by
$$ \mathcal{R}_{\Pi}(A,C) = \Pi ( L_N(A,C) ), \quad A \in \mathcal{A}, C \in \mathcal{C},$$
where $\Pi$ is a classical risk measure. This notion makes explicit the contribution of space in the risk measurement, and, for many useful risk measures $\Pi$ such as, e.g., variance, VaR and ES, enables one to take (at least) part} of the spatial dependence structure of the field $C$ into account. For a given $\Pi$ and $C \in \mathcal{C}$, the quantity $\mathcal{R}_{\Pi}(\cdot,C)$ is referred to as the spatial risk measure associated with $\Pi$ and induced by $C$. \Tb{The distribution of $L_N(A,C)$ only depends on $A$ and the finite-dimensional distributions of $C$ \citep[][Theorem 1]{koch2019SpatialRiskAxioms}}. Consequently, for a fixed $A$, if $\Pi$ is law-invariant, then $\mathcal{R}_{\Pi}(A, C)$ only depends on the finite-dimensional distributions of $C$.

\Tb{Now, let $\Pi$ be a classical risk measure, $C \in \mathcal{C}$ be fixed and, for $A \in \mathcal{A}$, let $\Mb{b}_A$ denote its barycenter. \cite{koch2017spatial, koch2019SpatialRiskAxioms} defined the following axioms} for the spatial risk measure associated with $\Pi$ and induced by $C$, $\mathcal{R}_{\Pi}(\cdot, C)$:
\Ben
\item Spatial invariance under translation: for all $\Mb{v} \in \mathbb{R}^2$ and $A \in \mathcal{A},  \ \mathcal{R}_{\Pi}(A+\Mb{v}, C)=\mathcal{R}_{\Pi}(A,C)$, where $A+\Mb{v}$ denotes the region $A$ translated by the vector $\Mb{v}$.
\item Spatial sub-additivity: for all $A_1, A_2 \in \mathcal{A},\  \mathcal{R}_{\Pi}(A_1 \cup A_2, C) \leq \min \{ \mathcal{R}_{\Pi}(A_1, C),\mathcal{R}_{\Pi}(A_2, C) \}$.
\item Asymptotic spatial homogeneity of order $- \gamma, \gamma \geq 0$: for all $A \in \mathcal{A}_c$,
$$
\mathcal{R}_{\Pi}(\lambda A, C) \underset{\lambda \to \infty}{=} K_1(A,C)+\dfrac{K_2(A,C)}{\lambda^{\gamma}} + o\left(\frac{1}{\lambda^{\gamma}}\right),
$$
where $\lambda A$ is the area obtained by applying to $A$ a homothety with center $\Mb{b}_A$ and ratio $\lambda >0$, and $K_1(\cdot, C): \mathcal{A}_c  \to \mathbb{R}$, $K_2(\cdot, C): \mathcal{A}_c  \to \mathbb{R} \backslash \{ 0 \}$ are functions depending on $C$.
\Een
It is also legitimate to introduce the axiom of spatial anti-monotonicity:
for all $A_1, A_2 \in \mathcal{A}$, $A_1 \subset A_2 \Rightarrow \mathcal{R}_{\Pi}(A_2, C) \leq \mathcal{R}_{\Pi}(A_1, C)$. The latter is equivalent to the axiom of spatial sub-additivity. \Tb{These axioms concern the spatial risk measures properties with respect to space and not to the cost distribution, the latter being fixed. They} seem natural and make sense at least under some conditions on the cost field $C$ and for some classical risk measures $\Pi$, as shown in \cite{koch2017spatial, koch2019SpatialRiskAxioms}. The axiom of spatial sub-additivity qualitatively points out spatial diversification. If it is satisfied \Tb{with strict inequality}, an insurance company would be well advised to underwrite policies in both regions $A_1$ and $A_2$ instead of only one of them. The axiom of asymptotic spatial homogeneity of order $-\gamma$ quantifies the rate of spatial diversification when the area becomes wide. Hence, knowing $\gamma$ might be valuable for the insurance industry. 

\Tb{For more details about these notions and an account of their value for real actuarial practice, see \cite{koch2019SpatialRiskAxioms}, Sections 2.1 and 2.2.}

\subsection{\Tb{Cost field model for extreme wind speeds}}
\label{Subsec_Specification_Cost_Field_Model}

\Tb{Throughout the paper, we assume the cost to be an insured cost.}
The general cost field model introduced in \cite{koch2017spatial}, Section 2.3, is defined by
$ \{ C(\Mb{x}) \}_{\Mb{x} \in \mathbb{R}^2} = \{ E(\Mb{x})\  D ( Z(\Mb{x})) \}_{\Mb{x} \in \mathbb{R}^2}$, where $\{ E(\Mb{x}) \}_{\Mb{x} \in \mathbb{R}^2}$ is the exposure field, $D$ the damage function and $\{ Z(\Mb{x}) \}_{\Mb{x} \in \mathbb{R}^2}$ the random field of the environmental variable generating risk. \Tb{Here, the cost is assumed to result from a unique wind hazard (e.g., \Tb{windstorms, hurricanes, tornadoes}) which is characterized by the random field of wind speed extremes over the period $[0, T_L]$, $Z$.
The application of the damage function $D$ to $Z$ yields the insured cost ratio at each site, which, multiplied by the exposure, gives the corresponding insured cost.} For the purpose of this paper, we choose the exposure to be uniformly equal to unity. 

\medskip

\Tb{We consider, for $\beta \in \mathbb{N}_*$ and $c_1 >0$, the damage function $D(z)=(z/c_1)^{\beta}, z \in \mathbb{R}$,
which is perfectly suited to the case of wind. Based on physical considerations, the total cost for a specific structure is expected to increase as the square or the cube of the maximum wind speed. Indeed, wind loads and dissipation rate of wind kinetic energy are proportional to the second and third powers of wind speed, respectively. For the square, see, e.g., \citet{simiu1996wind}, Equations (4.7.1), (8.1.1) and (8.1.8) and the interpretation following Equation (4.1.20). Regarding the third power, see, among others, \citet[][Chapter 2, p.7]{lamb1991historic}, where the cube of the wind speed appears in the severity index, and \cite{emanuel2005increasing}. In his discussion of the paper by \cite{powell2007tropical}, \cite{kantha2008tropical} states that wind damage for a given structure must be proportional to the rate of work done (and not the force exerted) by the wind and therefore strongly argues in favour of the cube rather than the square. In addition to this debate about whether the square or cube is more appropriate for total costs, several studies in the last two decades have found power-laws with much higher exponents when insured costs are considered. For instance, \cite{prahl2012applying} found powers ranging from $8$ to $12$ for insured losses on residential buildings in Germany (local damage functions). \cite{prahl2015comparison} argue that, if the total cost follows a cubic law but the insurance contract is triggered only when that cost exceeds a positive threshold (e.g., in the case of a contract with deductible), then the resulting cost for the insurance company is of power-law type but with a higher exponent. We have checked this statement using simulations and observed that the resulting exponent depends on the threshold (not shown).}

\Tb{Several authors \citep[e.g.,][]{klawa2003model, pinto2007changing, donat2011high} use, even in the case of insured losses, a cubic relationship that they justify with the physical arguments given above. However, they apply the third power to the difference between the wind speed value and a high percentile of the wind distribution and not to the effective wind speed; as shown by \citet[][Appendix A3]{prahl2015comparison}, this is equivalent to applying a much higher power to the effective wind speed.}

\Tb{Due to the various possible values for the right exponent in the damage function (especially depending on the value of the deductible), in this paper we will consider $\beta= 1, \ldots, 12$. Without loss of generality, we take $c_1=1$, which is consistent with the value we have chosen for the exposure. Note that exponential damage functions are sometimes also encountered in the literature \cite[e.g.,][]{huang2001long, prettenthaler2012risk}; we do not consider such functions here.}

\medskip

Furthermore, we take $Z$ to be a \Tb{max-stable random field} such that the field $Z^{\beta}$ belongs to $\mathcal{C}$, i.e., is measurable and has a.s. locally integrable sample paths. The latter property is satisfied, e.g., as soon as $Z$ is measurable and the function $\Mb{x} \mapsto  \mathbb{E} \left[ \left|Z(\Mb{x})^{\beta}\right| \right]$ is locally integrable \citep[][Proposition 1]{koch2019SpatialRiskAxioms}. Most often, $Z$ will be the Brown--Resnick random field \Tb{(with appropriate margins)}; see below. \Tb{As stated in, e.g., \citet[][Section 2.3]{huser2014space}, in addition to be very natural models for pointwise maxima, max-stable fields provide appropriate models for extremes of individual observations.
}

\Tb{
We shall sometimes assume that $Z$ has standard Fr\'echet margins; a max-stable field with such margins is said to be simple and will be indicated with ``$(\mathrm{s})$'' in superscript. Any simple max-stable random field $Z^{(\mathrm{s})}$ on $\mathbb{R}^d$ can be written \citep[e.g.,][]{haan1984spectral} as
\Beq
\label{Eq_Spectral_Representation_Stochastic_Processes}
\left \{ Z^{(\mathrm{s})}(\Mb{x}) \right \}_{\Mb{x} \in \mathbb{R}^d} \overset{d}{=} \left \{ \bigvee_{i=1}^{\infty} \{ U_i Y_i(\Mb{x}) \} \right \}_{\Mb{x} \in \mathbb{R}^d},
\Eeq
where the $(U_i)_{i \geq 1}$ are the points of a Poisson point process on $(0, \infty)$ with intensity function $u^{-2} \nu(\mathrm{d}u)$ and the $Y_i, i\geq 1$, are independent replications of a random field $\{ Y(\Mb{x}) \}_{\Mb{x} \in \mathbb{R}^d}$ such that, for all $\Mb{x} \in \mathbb{R}^d$,
$\mathbb{E}[Y(\Mb{x})]=1$. The field $Y$ is not unique and is called a spectral random field of $Z^{(\mathrm{s})}$. Conversely, any random field of the form \eqref{Eq_Spectral_Representation_Stochastic_Processes} is a simple max-stable field. Now, let $(U_i, \Mb{C}_i)_{i \geq 1}$ be the points of a Poisson point process on $(0,\infty) \times \mathbb{R}^d$ with intensity function $u^{-2} \nu(\mathrm{d}u) \times \nu(\mathrm{d}\Mb{c})$. Independently, let $f_i, i \geq 1$, be independent replicates of some non-negative random function $f$ on $\mathbb{R}^d$ satisfying $\mathbb{E} \left[ \int_{\mathbb{R}^d} f(\Mb{x}) \ \nu(\mathrm{d}\Mb{x}) \right]=1$. Then, the mixed moving maxima (M3) random field
\Beq
\label{Eq_Mixed_Moving_Maxima_Representation}
\{ Z^{(\mathrm{s})}(\Mb{x}) \}_{\Mb{x} \in \mathbb{R}^d}= \left \{ \bigvee_{i=1}^{\infty} \{ U_i f_i(\Mb{x}-\Mb{C}_i) \} \right \}_{\Mb{x} \in \mathbb{R}^d}
\Eeq
is a stationary\footnote{Throughout the paper, stationarity refers to strict stationarity.} and simple max-stable field. 
Equations \eqref{Eq_Spectral_Representation_Stochastic_Processes} and \eqref{Eq_Mixed_Moving_Maxima_Representation} are useful in practice as they enable the building up of parametric models for max-stable fields, some of which are briefly presented below. Let $\{ \varepsilon(\Mb{x}) \}_{\Mb{x} \in \mathbb{R}^d}$ be a stationary standard Gaussian random field with any correlation function.
}

\Tb{
Let $\{ W(\Mb{x}) \}_{\Mb{x} \in \mathbb{R}^d}$ be a centred Gaussian random field with stationary increments and with variogram $\gamma_W$, and $\{ Y(\Mb{x}) \}_{\Mb{x} \in \mathbb{R}^d}$ be defined by
$ Y(\Mb{x})=\exp \left( W(\Mb{x})-\mathrm{Var}(W(\Mb{x}))/2 \right)$,
where $\mathrm{Var}$ denotes the variance.
Then the field $Z^{(\mathrm{s})}$ defined by \eqref{Eq_Spectral_Representation_Stochastic_Processes} with that $Y$ is referred to as the \emph{Brown--Resnick random field} associated with the variogram $\gamma_W$ \citep{brown1977extreme, kabluchko2009stationary}. It is stationary and its distribution only depends on the variogram \citep[][Theorem 2 and Proposition 11, respectively]{kabluchko2009stationary}. The special case where $W(\Mb{x})=\sigma \varepsilon(\Mb{x})$, $\sigma>0$, leads to the \Tb{so-called} \emph{geometric Gaussian random field} \citep[e.g.,][]{davison2012statistical}.
Now, if $Z^{(\mathrm{s})}$ is written as in \eqref{Eq_Mixed_Moving_Maxima_Representation} with $f$ being the density of a $d$-variate Gaussian random vector with mean $\Mb{0}$ and positive-definite covariance matrix $\Sigma$, it is referred to as the \emph{Smith random field} with covariance matrix $\Sigma$ \citep{smith1990max}. The \emph{Schlather model} \citep{schlather2002models} results from taking $Y(\Mb{x}) = \sqrt{2 \pi} \varepsilon(\Mb{x})$ in \eqref{Eq_Spectral_Representation_Stochastic_Processes}. Finally, the \emph{tube model} (preliminary version of \cite{ancona2002diagnostics}, and \cite{koch2017spatial}) arises when taking in \eqref{Eq_Mixed_Moving_Maxima_Representation} $f(\Mb{y})=h_b \ \mathbb{I}_{ \{ \| \Mb{y} \| < R_b \} }, \Mb{y} \in \mathbb{R}^2$, where $R_b>0$ and $h_b=1/(\pi R_b^2)$.
}

\Tb{A commonly used variogram for the Brown--Resnick field is 
\Beq
\label{Eq_Power_Variogram}
\gamma_W(\Mb{x})= \left( \| \Mb{x} \| / \kappa \right)^{\psi}, \quad \Mb{x} \in \mathbb{R}^d,
\Eeq
where $\kappa>0$ and $\psi \in (0, 2]$ are the range and the smoothness parameters, respectively. An equivalent parametrization of \eqref{Eq_Power_Variogram} is $\gamma_W(\Mb{x}) = m \| \Mb{x} \|^{\psi}$, where $m>0$ and $\psi \in (0, 2]$.}
The variogram \Tb{$\gamma_W$ of a random field $W$} on $\mathbb{R}^d$ with stationary increments is said to be isotropic if, for all $\Mb{x} \in \mathbb{R}^d$, \Tb{$\gamma_W(\Mb{x})$} only depends on $\| \Mb{x} \|$, where $\| . \|$ denotes the Euclidean \Tb{norm}. In this case, we associate with \Tb{$\gamma_W$ the univariate function $\gamma_{W, \mathrm{u}}: [0, \infty) \to [0, \infty)$ such that, for all $\Mb{x} \in \mathbb{R}^d$, $\gamma_W(\Mb{x})=\gamma_{W, \mathrm{u}}(\| \Mb{x} \|)$}. \Tb{In the following we shall mainly focus on the class of Brown--Resnick random fields, which includes the Smith random field. Indeed the Smith field with covariance matrix $\Sigma$} corresponds to the Brown--Resnick field associated with the variogram 
\Beq
\label{Eq_Variogram_Smith_Field}
\gamma_W(\Mb{x})=\Mb{x}' \Sigma^{-1} \Mb{x}, \quad \Mb{x} \in \mathbb{R}^d,
\Eeq
where $'$ designates transposition; see, e.g., \cite{huser2013composite}. \Tb{The variogram in \eqref{Eq_Variogram_Smith_Field} can also be written $\| \Mb{x} \|_{\Sigma}^2$, where $\| \cdot \|_{\Sigma}$ is the norm associated with the inner product induced by the matrix $\Sigma^{-1}$.}

\Tb{The bivariate extremal coefficient function $\Theta$ \citep[e.g.,][]{schlather2003dependence}, which is a well-known measure of spatial dependence for max-stable fields, satisfies, for all $u >0$,
$\mathbb{P}\left( Z^{(\mathrm{s})}(\Mb{x}_1) \leq u, Z^{(\mathrm{s})}(\Mb{x}_2) \leq u \right) = \exp ( -\Theta(\Mb{x}_1, \Mb{x}_2)/u )$, $\Mb{x}_1, \Mb{x}_2 \in \mathbb{R}^d$, where $\{ Z^{(\mathrm{s})}(\Mb{x}) \}_{\Mb{x} \in \mathbb{R}^d}$ is simple max-stable.}

\Tb{
In practical applications, max-stable fields are not simple but have GEV univariate marginal distributions with location, scale and shape parameters $\eta \in \mathbb{R}$, $\tau >0$ and $\xi \in \mathbb{R}$. If $\{ Z(\Mb{x}) \}_{\Mb{x} \in \mathbb{R}^2}$ is a max-stable field with such GEV parameters, we can write
\Beq
\label{Eq_Link_Maxstb_Simple_Maxstab}
Z(\Mb{x}) = 
\left \{
\begin{array}{ll}
\left( \eta-\tau/\xi \right) + \tau Z^{(\mathrm{s})}(\Mb{x})^{\xi}/\xi, & \quad \xi \neq 0, \\
\eta + \tau \log(Z^{(\mathrm{s})}(\Mb{x})), & \quad \xi = 0,
\end{array}
\right.
\Eeq
where $\{ Z^{(\mathrm{s})}(\Mb{x}) \}_{\Mb{x} \in \mathbb{R}^d}$ is simple max-stable.
}

\medskip

\Tb{Finally, note that other authors considered a similar quantity as \eqref{Eq_NormSpatAggreLoss} within an extreme-value framework, but without cost field and not as a tool to develop a concept of spatial risk measures. For instance, \cite{coles1996modelling} modelled the so-called areal rainfall using the normalized spatial integral of a max-stable field, \cite{ferreira2012exceedance} investigated the tail properties of the integral over a compact region of a continuous random field in the maximum-domain of attraction of a max-stable field, and \cite{dombry2015functional} proposed the use of the spatial integral of a field to define threshold exceedances in the context of Pareto fields.}

\section{\Tb{Correlation of powers of Brown--Resnick fields and applications to wind extremes}}
\label{Sec_Dependence_Power_Maxstable}

\subsection{\Tb{Theory}}

Several dependence measures for max-stable random fields have been introduced in the literature: the extremal coefficient \citep[e.g.,][]{schlather2003dependence}, the F-madogram \citep{cooley2006variograms} and the $\lambda$-madogram \citep{naveau2009modelling}, among many others. \Tb{Here we propose a new spatial dependence measure which is the correlation of powers of max-stable fields and not max-stable fields themselves. More precisely, letting $\{ Z(\Mb{x}) \}_{\Mb{x} \in \mathbb{R}^2}$ be a max-stable random field with GEV parameters $\eta \in \mathbb{R}$, $\tau>0$, $\xi \neq 0$, and $\beta \in \mathbb{N}_*$ such that $\beta \xi < 1/2$, we focus on 
\Beq
\label{Eq_DepMeas}
\mathcal{D}_{\beta, \eta, \tau, \xi}(\Mb{x}_1, \Mb{x}_2)=\mbox{Corr} \left( Z(\Mb{x}_1)^{\beta}, Z(\Mb{x}_2)^{\beta} \right), \quad \Mb{x}_1, \Mb{x}_2 \in \mathbb{R}^2.
\Eeq
\Tb{The main motivation for considering this quantity lies in the fact that it can be seen as a measure of spatial dependence for damage due to wind (see Section \ref{Subsec_Specification_Cost_Field_Model}) and is thus fruitful for actuarial practice. It may also prove to be useful for the theoretical understanding of max-stable fields and it will be helpful} for the study of spatial risk measures associated with variance in Section \ref{Subsec_Variance}. Despite its drawbacks, correlation is commonly used in the finance/insurance industry, making its study useful from a practical point of view. Furthermore, the criticism that it does not properly capture extremal dependence is somehow irrelevant here as we consider the correlation between random variables which already model extreme events.} 

The purpose of this section is the study of $\mathcal{D}_{\beta, \eta, \tau, \xi}$ in the case of the Brown--Resnick random field. 
\Tb{Owing to the wide range of possible values for the damage power $\beta$, its sensitivity with respect to $\beta$ will also be considered.
The Brown--Resnick field is one of the most (if not the most) suitable models among currently available max-stable models, at least for environmental data \citep[e.g.,][Section 7.4, in the case of rainfall]{davison2012statistical}. Notably, it allows realistic realizations as well as independence when distance goes to infinity.}

\medskip

\Tb{From \eqref{Eq_Link_Maxstb_Simple_Maxstab}, we know that any max-stable field with general GEV margins can be expressed as a function of a simple max-stable field. Hence, we first consider simple max-stable fields as they are easier to handle. The following lemma is immediate and thus the proof is omitted.}
\begin{Lem}
\label{Chapriskmeasure_Lem_Frechet_Margins}
Let $\beta \in \mathbb{R}$ and $Z^{(\mathrm{s})}$ be a random variable following the standard Fr\'echet distribution. Then $(Z^{(\mathrm{s})})^{\beta} $ has a finite first moment if and only if (iff) $\beta < 1$ and a finite second moment iff $\beta < 1/2$.
Moreover,
$\mathbb{E} [(Z^{(\mathrm{s})})^{\beta}]= \Gamma(1-\beta)$, where $\Gamma$ denotes the gamma function.
\end{Lem}
For $\beta_1 , \beta_2 <1/2$, we introduce the function $g_{\beta_1, \beta_2}^{(\mathrm{s})}$ defined by
\Beq
\label{Eq_Def_g_beta1_beta2}
g_{\beta_1, \beta_2}^{(\mathrm{s})}(h) =
\left \{
\begin{array}{ll}
\Gamma(1-\beta_1-\beta_2), & \mbox{if} \quad  h=0, \\ 
\displaystyle \int_{0}^{\infty} \theta^{\beta_2} \Big[ C_2(\theta,h) \  C_1(\theta,h)^{\beta_1+\beta_2 -2} \ \Gamma(2-\beta_1-\beta_2) \\ \qquad + C_3(\theta,h) \ C_1(\theta,h)^{\beta_1+\beta_2-1} \ \Gamma(1-\beta_1-\beta_2) \Big] \  \nu(\mathrm{d}\theta), & \mbox{if} \quad h>0,
\end{array}
\right.
\Eeq
where, for $\theta, h > 0$,
\begin{align*}
C_1(\theta,h) &=   \Phi \left( \frac{h}{2}+ \frac{\log(\theta)}{h} \right)+\frac{1}{\theta} \Phi \left( \frac{h}{2}- \frac{\log(\theta)}{h} \right), \\
C_2(\theta,h) &= \left[   \Phi \left( \frac{h}{2}+ \frac{\log \left( \theta \right)}{h} \right) +\frac{1}{h} \phi \left( \frac{h}{2}+ \frac{\log(\theta)}{h} \right)-\frac{1}{h \theta} \phi \left(  \frac{h}{2}-\frac{\log \left( \theta \right)}{h} \right) \right]
\\& \quad \ \times \left[ \frac{1}{\theta^2} \Phi \left(  \frac{h}{2}- \frac{\log(\theta)}{h} \right)+\frac{1}{h \theta^2} \phi \left( \frac{h}{2}- \frac{\log(\theta)}{h} \right)-\frac{1}{h \theta} \phi \left( \frac{h}{2}+ \frac{\log(\theta)}{h} \right) \right], \\
C_3(\theta,h) &= \frac{1}{h^2 \theta} \left( \frac{h}{2}- \frac{\log(\theta)}{h} \right) \ \phi \left( \frac{h}{2}+ \frac{\log(\theta)}{h} \right)+\frac{1}{h^2 \theta^2} \left( \frac{h}{2}+ \frac{ \log(\theta) }{h} \right) \phi \left( \frac{h}{2}- \frac{\log \left( \theta \right)}{h}  \right),
\end{align*}
with $\Phi$ and $\phi$ \Tb{denoting the standard Gaussian distribution and density functions, respectively. We denote by $\mathrm{Cov}$ the covariance. The following result will help us a lot to derive the expression of $\mathcal{D}_{\beta, \eta, \tau, \xi}$.
\begin{Th}
\label{Th_Expression_Covariance_Damages}
Let $\{ Z^{(\mathrm{s})}(\Mb{x}) \}_{\Mb{x} \in \mathbb{R}^2}$ be a simple Brown--Resnick random field associated with the variogram $\gamma_W$. Then, for all $\Mb{x}_1, \Mb{x}_2 \in \mathbb{R}^2$ and $\beta_1, \beta_2 < 1/2$, we have 
\Beq
\label{Eq_Covariance_Damages_BR}
\mathrm{Cov} \left( Z^{(\mathrm{s})}(\Mb{x}_1)^{\beta_1}, Z^{(\mathrm{s})}(\Mb{x}_2)^{\beta_2} \right)= g_{\beta_1, \beta_2}^{(\mathrm{s})} \left( \sqrt{\gamma_W(\Mb{x}_2-\Mb{x}_1)} \right)- \Gamma(1-\beta_1) \Gamma(1-\beta_2).
\Eeq
\end{Th}
}
\begin{proof}
\Tb{Let $\beta_1 , \beta_2 < 1/2$ and $\Mb{x} \in \mathbb{R}^2$.} 

First, we show the result in the case where $\Mb{x}_1=\Mb{x}_2=\Mb{x}$. Since $Z$ is simple max-stable, it follows from Lemma \ref{Chapriskmeasure_Lem_Frechet_Margins} that
\begin{align*}
\mathrm{Cov} \left( Z^{(\mathrm{s})}(\Mb{x}_1)^{\beta_1}, Z^{(\mathrm{s})}(\Mb{x}_2)^{\beta_2} \right) =\Gamma(1-\beta_1-\beta_2)-\Gamma(1-\beta_1)\Gamma(1-\beta_2) = g_{\beta_1, \beta_2}^{(\mathrm{s})}(0)-\Gamma(1-\beta_1)\Gamma(1-\beta_2),
\end{align*} 
\Tb{which yields \eqref{Eq_Covariance_Damages_BR} as $\gamma_W(\Mb{0})=0$.}

Now, we prove the result in the case where $\Mb{x}_1, \Mb{x}_2$ are distinct vectors of $\mathbb{R}^2$. We have
$$\mathbb{E} \left[ Z^{(\mathrm{s})}(\Mb{x}_1)^{\beta_1} Z^{(\mathrm{s})}(\Mb{x}_2)^{\beta_2} \right] = \displaystyle \int_{0}^{\infty} \displaystyle \int_{0}^{\infty} z_1^{\beta_1} z_2^{\beta_2} l(z_1, z_2) \ \nu(\mathrm{d}z_1) \ \nu(\mathrm{d}z_2),$$
where $l$ denotes the bivariate density of the \Tb{Brown--Resnick field $Z^{(\mathrm{s})}$} (at $\Mb{x}_1$ and $\Mb{x}_2$).
In order to take advantage of the radius/angle decomposition of multivariate extreme-value distributions, we \Tb{make the change of variable}
$$
\begin{pmatrix}
z_1 \\
z_2
\end{pmatrix}
=\begin{pmatrix}
u \\
\theta\  u
\end{pmatrix}
=\begin{pmatrix}
\Psi_1(u, \theta) \\
\Psi_2(u, \theta)
\end{pmatrix}
=\Psi(u, \theta).
$$
The corresponding Jacobian matrix is written
$$
J_{\Psi}(u, \theta)=
\begin{pmatrix}
1 & 0 \\
& \\
\theta & u
\end{pmatrix},
$$  
and its determinant is thus $\det(J_{\Psi}(u, \theta))=u$.
Therefore, introducing 
$$a(z_1,z_2)=z_1^{\beta_1} z_2^{\beta_2} l(z_1, z_2), \quad z_1, z_2 >0,$$ 
we have
\begin{align}
\mathbb{E} \left[ Z^{(\mathrm{s})}(\Mb{x}_1)^{\beta_1} Z^{(\mathrm{s})}(\Mb{x}_2)^{\beta_2} \right] & = \int_{0}^{\infty} \int_{0}^{\infty} a(z_1,z_2) \ \nu(\mathrm{d}z_1) \ \nu(\mathrm{d}z_2) \nonumber
\\& = \int \int_{\Psi^{-1}((0, \infty )^2)} a(\Psi(u, \theta)) \det(J_{\Psi}(u, \theta)) \ \nu(\mathrm{d}u) \ \nu(\mathrm{d}\theta) \nonumber
\\&  = \int_{0}^{\infty} \int_{0}^{\infty} u^{\beta_1} \theta^{\beta_2} u^{\beta_2} l(u, \theta u) u \  \nu(\mathrm{d}u) \ \nu(\mathrm{d}\theta) \nonumber
\\&  = \int_{0}^{\infty} \int_{0}^{\infty} u^{\beta_1+\beta_2 +1 } \theta^{\beta_2} l(u, \theta u) \ \nu(\mathrm{d}u) \  \nu(\mathrm{d}\theta).
\label{Chapter_Riskmeasures_Eq1_Exp_Damage}
\end{align}
\Tb{Let $h_{\mathrm{S}}=\sqrt{(\Mb{x}_2-\Mb{x}_1)^{'} \Sigma^{-1} (\Mb{x}_2-\Mb{x}_1)}= \| \Mb{x}_2 - \Mb{x}_1 \|_{\Sigma}$.}
Equation (4) in \cite{padoan2010likelihood} gives that the bivariate density of the Smith random field (at $\Mb{x}_1$ and $\Mb{x}_2$) satisfies, for $z_1, z_2>0$,
\begin{align}
l_{\mathrm{S}}(z_1, z_2) =\exp \left( -\frac{\Phi(w)}{z_1}-\frac{\Phi(v)}{z_2} \right) \times \bigg[ \left( \frac{\Phi(w)}{z_1^2}+\frac{\phi(w)}{h z_1^2}-\frac{\phi(v)}{h z_1 z_2} \right)  &\times \left( \frac{\Phi(v)}{z_2^2}+\frac{\phi(v)}{h z_2^2}-\frac{\phi(w)}{h z_1 z_2} \right) \nonumber
 \\&  +\left( \frac{v \phi(w)}{h^2 z_1^2 z_2}+\frac{w \phi(v)}{h^2 z_1 z_2^2} \right) \bigg],
\label{Chapter_RiskMeasure_Smith_Bivariate_Density}
\end{align}
where 
\Tb{
$$w=\frac{h_{\mathrm{S}}}{2}+\frac{\log \left( z_2 / z_1 \right)}{h_{\mathrm{S}}} 
\quad \mbox{and} \quad v=\frac{h_{\mathrm{S}}}{2}-\frac{\log \left( z_2/z_1 \right)}{h_{\mathrm{S}}}.$$
}
\Tb{
It is known that the bivariate distribution function (at $\Mb{x}_1, \Mb{x}_2 \in \mathbb{R}^2$) of the Brown--Resnick random field associated with the variogram $\gamma_W$ is the same as that of the Smith random field with covariance matrix $\Sigma$ when replacing $\| \Mb{x}_2 - \Mb{x}_1 \|_{\Sigma}$ with $\sqrt{\gamma_W(\Mb{x}_2-\Mb{x}_1)}$; compare Equation (1) in \cite{huser2013composite} and Equation (3) in \cite{padoan2010likelihood}. It follows that the bivariate density (at $\Mb{x}_1$ and $\Mb{x}_2$) of the Brown--Resnick field associated with the variogram $\gamma_W$ is given by the right-hand side of \eqref{Chapter_RiskMeasure_Smith_Bivariate_Density} with $h_{\mathrm{S}}$ being replaced with $h=\sqrt{\gamma_W(\Mb{x}_2-\Mb{x}_1)}$.
Therefore,} for any $z_1, \theta>0$,
\begin{align}
& \quad \ l(z_1, \theta z_1) \nonumber \\
& = \exp \left( - \frac{1}{z_1} \left[ \Phi \left( \frac{h}{2}+ \frac{\log(\theta)}{h} \right)+\frac{1}{\theta} \Phi \left( \frac{h}{2} - \frac{\log(\theta)}{h} \right) \right] \right)
\times  \bigg \{ \frac{1}{z_1^4} \bigg[   \Phi \left( \frac{h}{2}+\frac{\log(\theta)}{h} \right) +\frac{1}{h} \phi \left( \frac{h}{2}+ \frac{\log(\theta)}{h} \right) \nonumber
\\& \ \ \  -\frac{1}{h \theta}\phi \left(  \frac{h}{2}-\frac{\log(\theta)}{h} \right) \bigg]
\times \left[ \frac{1}{\theta^2} \Phi \left(  \frac{h}{2}- \frac{\log(\theta)}{h} \right)+\frac{1}{h \theta^2} \phi \left( \frac{h}{2}- \frac{\log(\theta)}{h} \right)-\frac{1}{h \theta} \phi \left( \frac{h}{2}+ \frac{\log(\theta)}{h} \right) \right] \nonumber
\\& \ \ \  + \frac{1}{z_1^3} \left[ \frac{1}{h^2 \theta} \left( \frac{h}{2}- \frac{\log(\theta)}{h} \right) \ \phi \left( \frac{h}{2}+ \frac{\log(\theta)}{h} \right)+\frac{1}{h^2 \theta^2} \left( \frac{h}{2}+ \frac{\log(\theta)}{h} \right) \ \phi \left( \frac{h}{2}-\frac{\log \left( \theta \right)}{h}  \right) \right] \bigg \} \nonumber
\\& =  \exp \left( - \frac{C_1(\theta,h)}{z_1} \right) \left( \frac{C_2(\theta,h)}{z_1^4} + \frac{C_3(\theta,h)}{z_1^3} \right).
\label{Eq_h_u_tu}
\end{align}
\Tb{We denote by $\mathcal{F}_{s_f}$ the Fr\'echet distribution with shape and scale parameters $1$ and $s_f>0$, i.e., if $X \sim \mathcal{F}_{s_f}$,
$\mathbb{P}(X \leq x)= \exp ( - s_f/x ), x > 0.$
Using \eqref{Chapter_Riskmeasures_Eq1_Exp_Damage} and \eqref{Eq_h_u_tu} and the fact that the density of $X \sim \mathcal{F}_{s_f}$ is written
$l_f(x)=s_f/x^2 \exp \left( -s_f/x \right)$}, we obtain
\begin{align}
&\quad \ \mathbb{E} \left[ Z^{(\mathrm{s})}(\Mb{x}_1)^{\beta_1} Z^{(\mathrm{s})}(\Mb{x}_2)^{\beta_2} \right] \nonumber \\
&= \int_{0}^{\infty} \theta^{\beta_2} \left( \int_{0}^{\infty}  u^{\beta_1+\beta_2 +1} \exp \left( - \frac{C_1(\theta,h)}{u} \right) \left( \frac{C_2(\theta,h)}{u^4} + \frac{C_3(\theta,h)}{u^3} \right) \ \nu(\mathrm{d}u) \right) \nu(\mathrm{d}\theta) \nonumber
\\& = \int_{0}^{\infty} C_2(\theta,h) \ \theta^{\beta_2} \left( \int_{0}^{\infty}  u^{\beta_1+\beta_2 -3 } \exp \left( - \frac{C_1(\theta,h)}{u} \right) \  \nu(\mathrm{d}u) \right) \nu(\mathrm{d}\theta) \nonumber
\\& \ \ \  + \int_{0}^{\infty} C_3(\theta,h) \ \theta^{\beta_2} \left( \int_{0}^{\infty}  u^{\beta_1+\beta_2 -2 } \exp \left( - \frac{C_1(\theta,h)}{u} \right) \  \nu(\mathrm{d}u) \right) \nu(\mathrm{d}\theta) \nonumber
\\& = \int_{0}^{\infty} C_2(\theta,h) \ \theta^{\beta_2} \left( \int_{0}^{\infty}  u^{\beta_1+\beta_2 -1 }\  \frac{1}{u^2}\  \exp \left( - \frac{C_1(\theta,h)}{u} \right) \  \nu(\mathrm{d}u) \right) \nu(\mathrm{d}\theta) \nonumber
\\& \ \ \ + \int_{0}^{\infty} C_3(\theta,h) \ \theta^{\beta_2} \left( \int_{0}^{\infty}  u^{\beta_1+\beta_2} \ \frac{1}{u^2} \exp \left( - \frac{C_1(\theta,h)}{u} \right) \  \nu(\mathrm{d}u) \right) \nu(\mathrm{d}\theta) \nonumber
\\& = \int_{0}^{\infty} \frac{C_2(\theta,h)}{C_1(\theta,h)} \ \theta^{\beta_2} \ \mu_{ \beta_1 + \beta_2 -1} \left( \mathcal{F}_{C_1(\theta,h)} \right) \  \nu(\mathrm{d}\theta) + \int_{0}^{\infty} \frac{C_3(\theta,h)}{C_1(\theta,h)} \ \theta^{\beta_2} \ \mu_{ \beta_1 + \beta_2} \left( \mathcal{F}_{C_1(\theta,h)} \right) \  \nu(\mathrm{d}\theta),
\label{Eq_Expectation}
\end{align}
where $\mu_k(F)$ stands for the $k$-th moment of a random variable having $F$ as distribution. It is immediate to see that $\mu_k ( \mathcal{F}_{s_f} ) =s_f^k \  \Gamma(1-k)$,
which, combined with \eqref{Eq_Expectation}, yields the result.
\end{proof}

\medskip

\Tb{
We now aim at studying our spatial dependence measure $\mathcal{D}_{\beta, \eta, \tau, \xi}$. The next result, that we derive from Theorem \ref{Th_Expression_Covariance_Damages} constitutes an important step in that direction.
\begin{Th}
\label{Th_Cov_Maxstab_Real_Marg}
Let $\Mb{x}_1, \Mb{x}_2 \in \mathbb{R}^2$ and $\{ Z(\Mb{x}) \}_{\Mb{x} \in \mathbb{R}^2}$ be a Brown--Resnick field with GEV parameters $\eta_i \in \mathbb{R}$, $\tau_i>0$ and $\xi_i \neq 0$ at $\Mb{x}_i$, $i=1, 2$. Moreover, let $\beta_i \in \mathbb{N}_*$ such that $\beta_i \xi_i < 1/2$, $i=1,2$.
Then, we have
\begin{align*}
\mbox{Cov} \left( Z(\Mb{x}_1)^{\beta_1}, Z(\Mb{x}_2)^{\beta_2} \right)
&= \sum_{k_1=0}^{\beta_1} \sum_{k_2=0}^{\beta_2}  B_{k_1, \beta_1, \eta_1, \tau_1, \xi_1, k_2, \beta_2, \eta_2, \tau_2, \xi_2} \ g_{(\beta_1-k_1)\xi_1, (\beta_2-k_2) \xi_2}^{(\mathrm{s})} \left( \sqrt{\gamma_W(\Mb{x}_2-\Mb{x}_1)} \right)
\\& \quad - \sum_{k_1=0}^{\beta_1} \sum_{k_2=0}^{\beta_2} B_{k_1, \beta_1, \eta_1, \tau_1, \xi_1, k_2, \beta_2, \eta_2, \tau_2, \xi_2} \ \Gamma(1-[\beta_1-k_1]\xi_1) \Gamma(1-[\beta_2-k_2]\xi_2),
\end{align*}
where 
$$ B_{k_1, \beta_1, \eta_1, \tau_1, \xi_1, k_2, \beta_2, \eta_2, \tau_2, \xi_2} = {\beta_1 \choose k_1} \left( \eta_1-\frac{\tau_1}{\xi_1} \right)^{k_1} \left( \frac{\tau_1}{\xi_1} \right)^{\beta_1-k_1} {\beta_2 \choose k_2} \left( \eta_2-\frac{\tau_2}{\xi_2} \right)^{k_2} \left( \frac{\tau_2}{\xi_2} \right)^{\beta_2-k_2}.$$
\end{Th}
\begin{proof}
Using \eqref{Eq_Link_Maxstb_Simple_Maxstab} and the binomial theorem, we obtain
\begin{align*}
\mbox{Cov} \left( Z(\Mb{x}_1)^{\beta_1}, Z(\Mb{x}_2)^{\beta_2} \right) = \sum_{k_1=0}^{\beta_1} \sum_{k_2=0}^{\beta_2} & {\beta_1 \choose k_1} \left( \eta_1-\frac{\tau_1}{\xi_1} \right)^{k_1} \left( \frac{\tau_1}{\xi_1} \right)^{\beta_1-k_1} {\beta_2 \choose k_2} \left( \eta_2-\frac{\tau_2}{\xi_2} \right)^{k_2} \left( \frac{\tau_2}{\xi_2} \right)^{\beta_2-k_2} \nonumber \\
& \times \mbox{Cov}  \left(Z^{(\mathrm{s})}(\Mb{x}_1)^{(\beta_1-k_1) \xi_1},  Z^{(\mathrm{s})}(\Mb{x}_2)^{(\beta_2-k_2) \xi_2} \right),
\end{align*}
which directly yields the result by Theorem \ref{Th_Expression_Covariance_Damages}.
\end{proof}
The following is an immediate consequence of Theorem \ref{Th_Cov_Maxstab_Real_Marg}.
\begin{Corr}
Under the same assumptions as in Theorem \ref{Th_Cov_Maxstab_Real_Marg} but with $\eta_1=\eta_2=\eta$, $\tau_1=\tau_2=\tau$, $\xi_1=\xi_2=\xi$ and $\beta_1=\beta_2=\beta$, we have
\Beq
\mbox{Cov} \left( Z(\Mb{x}_1)^{\beta}, Z(\Mb{x}_2)^{\beta} \right)
= g_{\beta, \eta, \tau, \xi} \left( \sqrt{\gamma_W(\Mb{x}_2-\Mb{x}_1)} \right) - \sum_{k_1=0}^{\beta} \sum_{k_2=0}^{\beta} B_{k_1, k_2, \beta, \eta, \tau, \xi} \ \Gamma(1-[\beta-k_1]\xi) \Gamma(1-[\beta-k_2]\xi)
\label{Eq_Cov_Maxstab_Real_Marg_Eq_Coeff}
\Eeq
and
\Beq
\mbox{Var} \left( Z(\Mb{0})^{\beta} \right) = \sum_{k_1=0}^{\beta} \sum_{k_2=0}^{\beta} B_{k_1, k_2, \beta, \eta, \tau, \xi} \left \{ \Gamma(1-\xi[2 \beta - k_1 -k_2])- \Gamma(1-[\beta-k_1]\xi) \Gamma(1-[\beta-k_2]\xi) \right \},
\label{Eq_Var_GEV_beta}
\Eeq
where
$$ B_{k_1, k_2, \beta, \eta, \tau, \xi}= {\beta \choose k_1} {\beta \choose k_2} \left( \eta-\frac{\tau}{\xi} \right)^{k_1+k_2} \left( \frac{\tau}{\xi} \right)^{2\beta-(k_1+k_2)}$$
and
\Beq
g_{\beta, \eta, \tau, \xi}(h) = \sum_{k_1=0}^{\beta} \sum_{k_2=0}^{\beta} B_{k_1, k_2, \beta, \eta, \tau, \xi} \ g_{(\beta-k_1)\xi, (\beta-k_2) \xi}^{(\mathrm{s})} \left( h \right).
\label{Eq_Function_gtilde}
\Eeq
\end{Corr}
Our dependence measure $\mathcal{D}_{\beta, \eta, \tau, \xi}$ is given by the ratio of the right-hand-sides of \eqref{Eq_Cov_Maxstab_Real_Marg_Eq_Coeff} and \eqref{Eq_Var_GEV_beta}. 
In order to derive useful conclusions about $\mathcal{D}_{\beta, \eta, \tau, \xi}$, we investigate the behaviour of the function $g_{\beta, \eta, \tau, \xi}$ defined in \eqref{Eq_Function_gtilde}. For this purpose, we first need the following result.
}
\begin{Prop}
\label{Prop_Generalization_Dhaene}
For a random vector $\Mb{X}=(X_1, X_2)'$, its distribution function is denoted $F_{X_1, X_2}$. Let $\Mb{X}=(X_1, X_2)'$ and $\Mb{Y}=(Y_1, Y_2)'$ be random vectors having the same margins. Then, we have
$$ 
F_{X_1, X_2}(z_1, z_2) < F_{Y_1, Y_2}(z_1, z_2) \ \mbox{ for all } z_1, z_2 > 0 \Longrightarrow  \Tb{\mathrm{Cov}(f_1(X_1), f_2(X_2)) < \mathrm{Cov}(f_1(Y_1), f_2(Y_2))},
$$
for all \Tb{strictly increasing} functions $f_1: \Tb{(0, \infty)} \to \Tb{\mathbb{R}}$ and $f_2: \Tb{(0, \infty)} \to \Tb{\mathbb{R}}$, provided the covariances exist.
\end{Prop}
\begin{proof}
The proof is partly inspired from the proof of Theorem 1 in \cite{dhaene1996dependency}. Let $f_1: \Tb{(0, \infty)} \to \Tb{\mathbb{R}}$ and $f_2: \Tb{(0, \infty)} \to \Tb{\mathbb{R}}$ be \Tb{strictly increasing} functions. Assume that, for all $z_1, z_2>0$, 
\Beq
\label{Eq_Majoration_Function_Distribution}
F_{X_1, X_2}(z_1, z_2) < F_{Y_1, Y_2}(z_1, z_2).
\Eeq
We have 
$$ \mathbb{P}(f_1(X_1) \leq z_1, f_2(X_2) \leq z_2)= \mathbb{P} \left( X_1 \leq f_1^{-1}(z_1), X_2 \leq f_2^{-1}(z_2) \right)$$
and the same equality for $\Mb{Y}$. Consequently, since, for all $z_1, z_2>0$, $f_1^{-1}(z_1), f_2^{-1}(z_2)>0$, it follows from 
\eqref{Eq_Majoration_Function_Distribution} that, for all $z_1, z_2>0$,
\Beq
\label{Eq_Majoration_Function_Distribution_Transformed}
\mathbb{P}(f_1(X_1) \leq z_1, f_2(X_2) \leq z_2) < \mathbb{P}(f_1(Y_1) \leq z_1, f_2(Y_2) \leq z_2).
\Eeq
Since $X_1$ and $Y_1$ have the same distribution and the same is true for $X_2$ and $Y_2$, we deduce that
\Beq
\label{Eq_Same_Margins}
f_1(X_1) \overset{d}{=} f_1(Y_1) \quad \mbox{and} \quad f_2(X_2) \overset{d}{=} f_2(Y_2).
\Eeq
For a random variable $\tilde{X}$, we denote by $F_{\tilde{X}}$ its distribution function.
Using \eqref{Eq_Majoration_Function_Distribution_Transformed}, \eqref{Eq_Same_Margins} and  Lemma 1 in \cite{dhaene1996dependency}, we obtain
\begin{align*}
\mathrm{Cov}(f_1(X_1), f_2(X_2)) &=\int_{0}^{\infty} \int_{0}^{\infty} \left( F_{f_1(X_1), f_2(X_2)}(u,v)-F_{f_1(X_1)}(u) F_{f_2(X_2)}(v)) \  \nu(\mathrm{d}u \right) \ \nu(\mathrm{d}v) \\ 
& < \int_{0}^{\infty} \int_{0}^{\infty} \left( F_{f_1(Y_1), f_2(Y_2)}(u,v)-F_{f_1(Y_1)}(u) F_{f_2(Y_2)}(v)) \  \nu(\mathrm{d}u \right) \ \nu(\mathrm{d}v) \\
&= \mathrm{Cov}(f_1(Y_1), f_2(Y_2)).
\end{align*}
\end{proof}
Henceforth, we can show the \Tb{following technical result which will be used below.
\begin{Prop}
\label{Prop_Decrease_gtilde}
For all $\eta \in \mathbb{R}$, $\tau>0$, $\xi \neq 0$ and $\beta \in \mathbb{N}_*$ such that $\beta \xi < 1/2$, the function $g_{\beta, \eta, \tau, \xi}$ defined in \eqref{Eq_Function_gtilde} is strictly decreasing.
\end{Prop}
\begin{proof}
Let $\{ Z^{(\mathrm{s})}(\Mb{x}) \}_{\Mb{x} \in \mathbb{R}^2}$ be the simple Smith random field with covariance matrix $\Sigma$, \Tb{which is a Brown--Resnick random field associated with the variogram $\gamma_W(\Mb{x})= \| \Mb{x} \|_{\Sigma}$}. Let $\Mb{x}, \Mb{y} \in \mathbb{R}^2$.
Equation (3.1) in \cite{smith1990max} gives that, for all $z_1, z_2>0$,
$$ \mathbb{P}(Z^{(\mathrm{s})}(\Mb{x}) \leq z_1, Z^{(\mathrm{s})}(\Mb{y}) \leq z_2) =\exp \left( -\frac{1}{z_1} \Phi\left( \frac{h}{2} + \frac{1}{h} \log \left( \frac{z_2}{z_1} \right) \right) -\frac{1}{z_2} \Phi\left( \frac{h}{2} + \frac{1}{h} \log \left( \frac{z_1}{z_2} \right) \right) \right),$$
where $h=\| \Mb{y} - \Mb{x} \|_{\Sigma}$.
We immediately obtain that
\Tb{
\Beq
\label{Eq_Derivative_h_Bivariate_Df_Smith}
\frac{\partial \mathbb{P}(Z^{(\mathrm{s})}(\Mb{x}) \leq z_1, Z^{(\mathrm{s})}(\Mb{y}) \leq z_2)}{\partial h} = \exp \left( -\frac{1}{z_1} \Phi\left( \frac{h}{2} + \frac{1}{h} \log \left( \frac{z_2}{z_1} \right) \right) -\frac{1}{z_2} \Phi\left( \frac{h}{2} + \frac{1}{h} \log \left( \frac{z_1}{z_2} \right) \right) \right) T_2,
\Eeq
where
$$ T_2 = \left[-\frac{1}{z_1} \left(\frac{1}{2} - \frac{\log(z_2/z_1)}{h^2} \right) \phi \left( \frac{h}{2}+\frac{\log(z_2/z_1)}{h} \right) - \frac{1}{z_2} \left(\frac{1}{2} + \frac{\log(z_2/z_1)}{h^2} \right) \phi \left( \frac{h}{2}-\frac{\log(z_2/z_1)}{h} \right) \right].
$$
For all $z_1, z_2>0$, we introduce $y=z_2/z_1$, which is positive. We have
\begin{align*}
T_2 &= \frac{1}{z_2} \left[-\frac{z_2}{z_1} \left(\frac{1}{2} - \frac{\log(z_2/z_1)}{h^2} \right) \phi \left( \frac{h}{2}+\frac{\log(z_2/z_1)}{h} \right) - \left(\frac{1}{2} + \frac{\log(z_2/z_1)}{h^2} \right) \phi \left( \frac{h}{2}-\frac{\log(z_2/z_1)}{h} \right) \right]
\\& = \frac{1}{z_2} \left[-y \left(\frac{1}{2} - \frac{\log(y)}{h^2} \right) \phi \left( \frac{h}{2}+\frac{\log(y)}{h} \right) - \left(\frac{1}{2} + \frac{\log(y)}{h^2} \right) \phi \left( \frac{h}{2}-\frac{\log(y)}{h} \right) \right]
\\& = \frac{1}{\sqrt{2 \pi} z_2} \exp\left( -\frac{h^2}{8}-\frac{\log(y)^2}{2h^2} \right) \left[ -y \left(\frac{1}{2} - \frac{\log(y)}{h^2} \right)y^{-1/2} - \left(\frac{1}{2} + \frac{\log(y)}{h^2} \right) y^{1/2} \right]
\\& = -\frac{y^{1/2}}{\sqrt{2 \pi} z_2} \exp\left( -\frac{h^2}{8}-\frac{\log(y)^2}{2h^2} \right),
\end{align*}
which is negative. Thus, \eqref{Eq_Derivative_h_Bivariate_Df_Smith} gives that}, for all $\Mb{x}, \Mb{y} \in \mathbb{R}^2$ and $z_1, z_2>0$, 
\Beq
\label{Eq_Derivative_h_Bivariate_Df_Smith_Negative}
\partial \mathbb{P}(Z^{(\mathrm{s})}(\Mb{x}) \leq z_1, Z^{(\mathrm{s})}(\Mb{y}) \leq z_2)/\partial h <0.
\Eeq
Let us consider $h_1>h_2>0$ and $\Mb{x}_1$, $\Mb{x}_2$, $\Mb{x}_3$, $\Mb{x}_4 \in \mathbb{R}^2$ such that 
\Beq
\label{Eq_Def_h1_h2}
h_1=\| \Mb{x}_2-\Mb{x}_1 \|_{\Sigma} \quad \mbox{and} \quad h_2=\| \Mb{x}_4-\Mb{x}_3 \|_{\Sigma}.
\Eeq 
It follows from \eqref{Eq_Derivative_h_Bivariate_Df_Smith_Negative} that, for all $z_1, z_2 >0$,
$F_{Z^{(\mathrm{s})}(\Mb{x}_1), Z^{(\mathrm{s})}(\Mb{x}_2)}(z_1, z_2) < F_{Z^{(\mathrm{s})}(\Mb{x}_3), Z^{(\mathrm{s})}(\Mb{x}_4)}(z_1, z_2)$. 
Since $Z^{(\mathrm{s})}$ is simple max-stable, we have $F_{Z^{(\mathrm{s})}(\Mb{x}_1)}=F_{Z^{(\mathrm{s})}(\Mb{x}_3)}$ and $F_{Z^{(\mathrm{s})}(\Mb{x}_2)}=F_{Z^{(\mathrm{s})}(\Mb{x}_4)}$. 
Now, as $\tau>0$, for $\xi \neq 0$, the function 
$$
\begin{array}{cccc}
f: & (0, \infty) & \to & \mathbb{R} \\
 & z & \mapsto & \left[ \left( \eta-\tau/\xi \right) + \tau z^{\xi}/\xi \right]^{\beta}
\end{array}
$$
is strictly increasing. 
Hence, letting 
$$ Z(\Mb{x}) = \left( \eta-\frac{\tau}{\xi} \right) + \frac{\tau}{\xi} Z^{(\mathrm{s})}(\Mb{x})^{\xi}, \quad \Mb{x} \in \mathbb{R}^2,$$
Proposition \ref{Prop_Generalization_Dhaene} yields
\Beq
\label{Eq_Order_Expectation}
\mathrm{Cov} \left( Z(\Mb{x}_1)^{\beta}, Z(\Mb{x}_2)^{\beta} \right) < \mathrm{Cov} \left( Z(\Mb{x}_3)^{\beta}, Z(\Mb{x}_4)^{\beta} \right).
\Eeq
Furthermore, we know from \eqref{Eq_Cov_Maxstab_Real_Marg_Eq_Coeff} that, for all $\Mb{x}, \Mb{y} \in \mathbb{R}^2$ satisfying $\| \Mb{y}-\Mb{x} \|_{\Sigma}=h$, 
\Beq
\label{Eq_g_Expectation}
\mathrm{Cov} \left( Z(\Mb{x})^{\beta}, Z(\Mb{y})^{\beta} \right)=g_{\beta, \eta, \tau, \xi}(h)-\sum_{k_1=0}^{\beta} \sum_{k_2=0}^{\beta} B_{k_1, k_2, \beta, \eta, \tau, \xi} \ \Gamma(1-[\beta-k_1]\xi) \Gamma(1-[\beta-k_2]\xi).
\Eeq
Finally, the combination of \eqref{Eq_Def_h1_h2}, \eqref{Eq_Order_Expectation} and \eqref{Eq_g_Expectation} gives that
$g_{\beta, \eta, \tau, \xi}(h_1)<g_{\beta, \eta, \tau, \xi}(h_2)$, showing the result. 
\end{proof}
The two following propositions especially give the behaviour of the function $g_{\beta, \eta, \tau, \xi}$ around $0$ and at $\infty$.
\begin{Prop}
\label{Prop_Lim_gtildebeta_hto0}
For all $\eta \in \mathbb{R}$, $\tau>0$, $\xi \neq 0$ and $\beta \in \mathbb{N}_*$ such that $\beta \xi < 1/2$, the function $g_{\beta, \eta, \tau, \xi}$ defined in \eqref{Eq_Function_gtilde} satisfies
\Beq
\label{Eq_Lim_gtildebeta_hto0}
\lim_{h \to 0} g_{\beta, \eta, \tau, \xi}(h) = \sum_{k_1=0}^{\beta} \sum_{k_2=0}^{\beta} B_{k_1, k_2, \beta, \eta, \tau, \xi} \Gamma(1-\xi[2 \beta - k_1 -k_2])
\Eeq
and is continuous everywhere on $[0, \infty)$.
\end{Prop}
\begin{proof}
Let $\{ Z(\Mb{x}) \}_{\Mb{x} \in \mathbb{R}^2}$ be the Smith random field with covariance matrix $\Sigma=I_2$, where $I_2$ is the identity matrix in dimension 2, and with GEV parameters $\eta$, $\tau$ and $\xi$ . The field $Z^{\beta}$ is stationary by stationarity of $Z$. In addition, as $\beta \xi < 1/2$, we know from Lemma \ref{Chapriskmeasure_Lem_Frechet_Margins} that $Z^{\beta}$ has a finite second moment. Accordingly, $Z^{\beta}$ is second-order stationary. Moreover, since the Smith random field is sample-continuous, it immediately follows that $Z^{\beta}$ is sample-continuous and thus, using the same arguments as in the proof of Proposition 1 in \cite{koch2017TCL}, that it is continuous in quadratic mean. Hence, the covariance function of the field $Z^{\beta}$ is continuous at the origin. This implies, using \eqref{Eq_Cov_Maxstab_Real_Marg_Eq_Coeff}, that
\begin{align*}
\lim_{\Mb{x} \to \Mb{0}} \mathrm{Cov} \left( Z(\Mb{0})^{\beta}, Z(\Mb{x})^{\beta} \right) &=\lim_{\Mb{x} \to \Mb{0}} \left( g_{\beta, \eta, \tau, \xi} \left( \| \Mb{x} \| \right)- \sum_{k_1=0}^{\beta} \sum_{k_2=0}^{\beta} B_{k_1, k_2, \beta, \eta, \tau, \xi} \ \Gamma(1-[\beta-k_1]\xi) \Gamma(1-[\beta-k_2]\xi) \right) \nonumber
\\& = \mbox{Var} \left( Z(\Mb{0})^{\beta} \right), 
\end{align*}
which, combined with \eqref{Eq_Var_GEV_beta}, yields \eqref{Eq_Lim_gtildebeta_hto0}.
This easily gives $\lim_{h \to 0}  g_{\beta, \eta, \tau, \xi}(h)= g_{\beta, \eta, \tau, \xi}(0)$, which implies that $g_{\beta, \eta, \tau, \xi}$ is continuous at $h=0$. The continuity of $g_{\beta, \eta, \tau, \xi}$ at any $h>0$ comes from the fact that the covariance function of a field which is second-order stationary can be discontinuous only at the origin.
\end{proof}
\begin{Prop}
\label{Prop_Lim_gtilde_infty}
For all $\eta \in \mathbb{R}$, $\tau>0$, $\xi \neq 0$ and $\beta \in \mathbb{N}_*$ such that $\beta \xi < 1/2$, the function $g_{\beta, \eta, \tau, \xi}$ defined in \eqref{Eq_Function_gtilde} satisfies
\Beq
\label{Eq_Lim_gtilde_infty}
\lim_{h \to \infty} g_{\beta, \eta, \tau, \xi}(h)=\sum_{k_1=0}^{\beta} \sum_{k_2=0}^{\beta} B_{k_1, k_2, \beta, \eta, \tau, \xi} \ \Gamma(1-[\beta-k_1]\xi) \Gamma(1-[\beta-k_2]\xi).
\Eeq
\end{Prop}
\begin{proof}
Let $\{ Z(\Mb{x}) \}_{\Mb{x} \in \mathbb{R}^2}$ be the Smith random field with covariance matrix $\Sigma=I_2$, and with GEV parameters $\eta$, $\tau$ and $\xi$. As will be seen in the proof of Theorem \ref{TCL_Wind}, the field $Z^{\beta}$ satisfies the central limit theorem (in a sense specified in Section \ref{Subsec_CLT_Homothety}). This implies (see Section \ref{Subsec_CLT_Homothety}) that 
$$
\int_{\mathbb{R}^2} \left| \mathrm{Cov} \left( Z(\Mb{0})^{\beta}, Z(\Mb{x})^{\beta} \right) \right| \  \nu(\mathrm{d}\Mb{x}) < \infty,
$$
which entails, using \Tb{\eqref{Eq_Cov_Maxstab_Real_Marg_Eq_Coeff}}, that
$$ \int_{\mathbb{R}^2} \left( g_{\beta, \eta, \tau, \xi} \left( \| \Mb{x} \| \right) - \sum_{k_1=0}^{\beta} \sum_{k_2=0}^{\beta} B_{k_1, k_2, \beta, \eta, \tau, \xi} \ \Gamma(1-[\beta-k_1]\xi) \Gamma(1-[\beta-k_2]\xi) \right) \  \nu(\mathrm{d}\Mb{x}) < \infty.$$
Since $g_{\beta, \eta, \tau, \xi}$ is strictly decreasing, this necessarily implies that 
$$ \lim_{h \to \infty} \left( g_{\beta, \eta, \tau, \xi} \left( h \right) - \sum_{k_1=0}^{\beta} \sum_{k_2=0}^{\beta} B_{k_1, k_2, \beta, \eta, \tau, \xi} \ \Gamma(1-[\beta-k_1]\xi) \Gamma(1-[\beta-k_2]\xi) \right)=0,$$
i.e., \eqref{Eq_Lim_gtilde_infty}. 
\end{proof}
}
\Tb{We have assumed throughout this section that $\xi \neq 0$ but, as shown by the next proposition, the case $\xi=0$ is easily recovered by letting $\xi$ tend to $0$ in the expressions above.
\begin{Prop}
\label{Prop_Continuity_Cov_xi_0}
Let $\eta \in \mathbb{R}$, $\tau>0$, $\xi \neq 0$ and define
$$ Z_{\xi}(\Mb{x}) = \eta + \tau (Z^{(\mathrm{s})}(\Mb{x})^{\xi}-1)/\xi \quad \mbox{and} \quad Z_0(\Mb{x}) = \eta + \tau \log( Z^{(\mathrm{s})}(\Mb{x})), \quad \Mb{x} \in \mathbb{R}^2,$$ 
where $Z^{(\mathrm{s})}$ is a simple max-stable random field. Let $\beta \in \mathbb{N}_*$ such that there exists $\varepsilon >0$ satisfying $2 \beta \xi (1+\varepsilon) <1$.
Then, we have, for all $\Mb{x}_1, \Mb{x}_2 \in \mathbb{R}^2$,
$$ \lim_{\xi \to 0} \mathrm{Cov} \left( Z_{\xi}(\Mb{x}_1)^{\beta},  Z_{\xi}(\Mb{x}_2)^{\beta} \right) =  \mathrm{Cov} \left( Z_0(\Mb{x}_1)^{\beta},  Z_0(\Mb{x}_2)^{\beta} \right).$$
\end{Prop}
}
\begin{proof}
\Tb{
For $i=1,2$, $Z_{\xi}(\Mb{x}_i)$ follows the GEV with parameters $\eta$, $\tau$ and $\xi$, the density of which we denote by $f$. For $\xi<0$, we easily obtain
$ \mathbb{E}[| Z_{\xi}(\Mb{x}_i)|^{\alpha}] < \infty$ for any $\alpha >0$, and, for $\xi>0$, we have for all $\alpha>0$
$$\mathbb{E}[| Z_{\xi}(\Mb{x}_i)|^{\alpha}] = \int_{\eta-\tau/\xi}^{0} |x|^{\alpha} f(x) \ \nu(\mathrm{d}x) + \int_{0}^{\infty} x^{\alpha} f(x) \ \nu(\mathrm{d}x).$$
It is easy to see that the first integral is finite and that the only possible problem for the existence of the second one is at $\infty$. We have
$$ \int_{0}^{\infty} x^{\alpha} \exp \left( -[1+\xi (x-\eta)/\tau]^{-1/\xi} \right) [1+\xi (x-\eta)/\tau]^{-1/\xi-1} \nu(\mathrm{d}x)=\int_{0}^{1} \left[ \eta + \tau(z^{-\xi}-1)/\xi \right]^{\alpha} \exp(-z) \nu(\mathrm{d}z),$$
where we used the change of variable $z=[1+\xi (x-\eta)/\tau]^{-1/\xi}$.
As $[ \eta + \tau(z^{-\xi}-1)/\xi ] \underset{\xi \to 0}{\sim} \tau z^{-\xi}/\xi$, the previous integral is finite provided $\alpha \xi <1$,
which yields by assumption that $\mathbb{E}[ |Z_{\xi}(\Mb{x}_i)^{\beta}|^{1+\varepsilon}]< \infty$. Now, let $Y_{\xi}=Z_{\xi}(\Mb{x}_1)^{\beta} Z_{\xi}(\Mb{x}_2)^{\beta}$, $\xi \neq 0$, and $Y_{0}=Z_{0}(\Mb{x}_1)^{\beta} Z_{0}(\Mb{x}_2)^{\beta}$. 
By Cauchy-Schwarz inequality,
$$
\mathbb{E}\left[ \left| Y_{\xi} \right|^{1+\varepsilon} \right] \leq \sqrt{\mathbb{E} \left[ \left| Z_{\xi}(\Mb{x}_1) \right|^{2\beta(1+\varepsilon)}\right]} \sqrt{\mathbb{E} \left[ \left| Z_{\xi}(\Mb{x}_2) \right|^{2\beta(1+\varepsilon)}\right]},
$$
which gives $\mathbb{E}[ | Y_{\xi} |^{1+\varepsilon}]< \infty$. It follows from \citet[][p.31]{billingsley1999convergence} that the $(Z_{\xi}(\Mb{x}_1))_{\xi}$, $(Z_{\xi}(\Mb{x}_2))_{\xi}$ and $(Y_{\xi})_{\xi}$ are uniformly integrable for $\xi$ around $0$.}

\Tb{
Now, it is well-known that $Z_{\xi}(\Mb{x}_i) \overset{d}{\to} Z_{0}(\Mb{x}_i)$, $i=1, 2$, which implies by the continuous mapping theorem that $Z_{\xi}(\Mb{x}_i)^{\beta} \overset{d}{\to} Z_{0}(\Mb{x}_i)^{\beta}$. Moreover, for any $z_1, z_2 \in \mathbb{R}$,
\begin{align*}
\mathbb{P} \left( \left[Z^{(\mathrm{s})}(\Mb{x}_1)-1\right]/\xi \leq z_1, \left[Z^{(\mathrm{s})}(\Mb{x}_2)-1 \right]/\xi \leq z_2 \right) &= \mathbb{P} \left( Z^{(\mathrm{s})}(\Mb{x}_1) \leq (1+\xi z_1)^{1/\xi}, Z^{(\mathrm{s})}(\Mb{x}_2) \leq (1+\xi z_2)^{1/\xi} \right) \\
&= \exp \left( -V \left([1+\xi z_1]^{1/\xi}, [1+\xi z_2]^{1/\xi} \right) \right),
\end{align*}
and 
$$ \mathbb{P} \left( \log Z^{(\mathrm{s})}(\Mb{x}_1) \leq z_1, \log Z^{(\mathrm{s})}(\Mb{x}_2) \leq z_2 \right) = \exp(-V(\exp(z_1), \exp(z_2)),$$
where $V$ is the exponent function of $(Z^{(\mathrm{s})}(\Mb{x}_1), Z^{(\mathrm{s})}(\Mb{x}_2))'$.
Thus,
$$ \lim_{\xi \to 0} \mathbb{P} \left( \left[Z^{(\mathrm{s})}(\Mb{x}_1)-1\right]/\xi \leq z_1, \left[Z^{(\mathrm{s})}(\Mb{x}_2)-1\right]/\xi \leq z_2 \right)= 
\mathbb{P} \left( \log Z^{(\mathrm{s})}(\Mb{x}_1) \leq z_1, \log Z^{(\mathrm{s})}(\Mb{x}_2) \leq z_2 \right),$$
and thus $$\left( \left[Z^{(\mathrm{s})}(\Mb{x}_1)-1\right]/\xi, \left[Z^{(\mathrm{s})}(\Mb{x}_2)-1\right]/\xi \right)' \overset{d}{\to} \left( \log Z^{(\mathrm{s})}(\Mb{x}_1), \log Z^{(\mathrm{s})}(\Mb{x}_2) \right)'.$$
Consequently, the continuous mapping theorem yields 
$$(Z_{\xi}(\Mb{x}_1)^{\beta}, Z_{\xi}(\Mb{x}_2)^{\beta})' \overset{d}{\to} (Z_{0}(\Mb{x}_1)^{\beta}, Z_{0}(\Mb{x}_2)^{\beta})',$$
and hence, applied again, $Y_{\xi} \overset{d}{\to} Y_{0}$.
Finally, Theorem 3.5 in \cite{billingsley1999convergence} yields that
$\lim_{\xi \to 0} \mathbb{E}(Z_{\xi}(\Mb{x}_i)) = \mathbb{E}(Z_{0}(\Mb{x}_i))$, $i=1,2$ and $\lim_{\xi \to 0} \mathbb{E}(Y_{\xi})=\mathbb{E}(Y_{0})$. The result follows immediately.}
\end{proof}
\Tb{Using similar arguments, we can show that $\lim_{\xi \to 0} \mathrm{Var}(Z_{\xi}(\Mb{x}_i)^{\beta}=\mathrm{Var}(Z_{0}(\Mb{x}_i)^{\beta}$, $i=1,2$, which yields, for any $\Mb{x}_1, \Mb{x}_2 \in \mathbb{R}^2$,
$$ \lim_{\xi \to 0} \mathrm{Corr} \left( Z_{\xi}(\Mb{x}_1)^{\beta},  Z_{\xi}(\Mb{x}_2)^{\beta} \right) =  \mathrm{Corr} \left( Z_0(\Mb{x}_1)^{\beta},  Z_0(\Mb{x}_2)^{\beta} \right).$$
}

\subsection{\Tb{Application}}

\Tb{
The expression of $\mathcal{D}_{\beta, \eta, \tau, \xi}(\Mb{x}_1, \Mb{x}_2)$ depends on $\Mb{x}_1$ and $\Mb{x}_2$ through $\gamma_W(\Mb{x}_2 - \Mb{x}_1)$ only; accordingly, we use in the following  the notation $\mathcal{D}_{\beta, \eta, \tau, \xi}(\gamma_W(\Mb{x}_2-\Mb{x}_1))$. Below, we study the behaviour of $\mathcal{D}_{\beta, \eta, \tau, \xi}(\gamma_W(\Mb{x}_2-\Mb{x}_1))$ with respect to $\beta$ and the Euclidean distance $\| \Mb{x}_2-\Mb{x}_1 \|$ for different variograms. As a variogram is a non-negative conditionally negative definite function, it follows from \citet[][Chapter 4, Section 3, Proposition 3.3]{berg1984harmonic}\footnote{In that book, the term ``non negative'' is used for ``conditionally non negative''.} that $d(\Mb{x}_1, \Mb{x}_2)= \sqrt{\gamma_W(\Mb{x}_2-\Mb{x}_1)}$, $\Mb{x}_1, \Mb{x}_2 \in \mathbb{R}^2$, defines a metric. For many common models of isotropic variogram $\gamma_W$, $\gamma_W(\Mb{x}_2-\Mb{x}_1)$ is a strictly increasing function of $\| \Mb{x}_2-\Mb{x}_1 \|$, which implies by \eqref{Eq_Cov_Maxstab_Real_Marg_Eq_Coeff} and Proposition \ref{Prop_Decrease_gtilde} that $\mathcal{D}_{\beta, \eta, \tau, \xi}(\gamma_W(\Mb{x}_2-\Mb{x}_1))$ is a strictly decreasing function of $\| \Mb{x}_2-\Mb{x}_1 \|$; such a decrease of the correlation with the distance seems natural. Moreover, \eqref{Eq_Cov_Maxstab_Real_Marg_Eq_Coeff}, \eqref{Eq_Var_GEV_beta} and \eqref{Eq_Lim_gtildebeta_hto0} give that $\lim_{ \Mb{x}_2-\Mb{x}_1 \to \Mb{0}} \mathcal{D}_{\beta, \eta, \tau, \xi}(\gamma_W(\Mb{x}_2-\Mb{x}_1))=1$. Now, introducing $\mathcal{B}_1= \{ \Mb{x} \in \mathbb{R}^2: \| \Mb{x} \|=1 \}$, for a function $f$ from $\mathbb{R}^2$ to $\mathbb{R}$, $\lim_{\| \Mb{h} \| \to \infty} f(\Mb{h})=\infty$ must be understood as $\lim_{h \to \infty} \inf_{\Mb{u} \in \mathcal{B}_1} \{ f(h \Mb{u}) \}=\infty$. Using \eqref{Eq_Cov_Maxstab_Real_Marg_Eq_Coeff} and \eqref{Eq_Lim_gtilde_infty}, we deduce, provided that $\lim_{\| \Mb{x}_2-\Mb{x}_1 \| \to \infty} \gamma_W(\Mb{x}_2-\Mb{x}_1)=\infty$, that $\lim_{\| \Mb{x}_2-\Mb{x}_1 \| \to \infty}  \mathcal{D}_{\beta, \eta, \tau, \xi}(\gamma_W(\Mb{x}_2-\Mb{x}_1))=0$. Furthermore, the faster the increase of $\gamma_W$ to infinity, the faster the convergence of $\mathcal{D}_{\beta, \eta, \tau, \xi}(\gamma_W(\Mb{x}_2-\Mb{x}_1))$ to $0$. These results are consistent with our expectations.}

\Tb{
We now thoroughly study the evolution of $\mathcal{D}_{\beta, \eta, \tau, \xi}(\gamma_W(\Mb{x}_2-\Mb{x}_1))$ with respect to $\| \Mb{x}_2-\Mb{x}_1 \|$ and $\beta$, for fixed values of the GEV parameters $\eta$, $\tau$ and $\xi$. As we want our study to be closely linked to practice, we choose reasonable values of those GEV parameters. Fitting a generalized Pareto distribution (GPD) to in situ observations over Switzerland, \cite{ceppi2008extreme} obtained a shape parameter $\xi$ ranging from $-0.2$ to $0$. \cite{della2007extreme} fitted the GPD to ERA-40 reanalysis data over Europe during windstorms and also found negative shape parameters, with values between $-0.1$ and $-0.3$ on most of land areas; see their Figure 4.15. More generally, many studies point out that $\xi$ is generally slightly negative, entailing that the distribution of wind speed maxima has a finite right endpoint; when a positive value is obtained, it is never too far from $0$. In the case of annual maxima over Europe, classic values for the location parameter $\eta$ lie between $25$ and $30$ m.s$^{-1}$ and usual values of the scale parameter $\tau$ range from $2.5$ to $3.5$ m.s$^{-1}$. E.g., considering annual maxima wind speeds at 35 weather stations in the Netherlands, \cite{ribatet2013spatial} obtained trend surfaces whose intercepts are about $27$ m.s$^{-1}$ for $\eta$ and $3.25$ m.s$^{-1}$ for $\tau$. We considered the $10$ m wind gust ERA-5 reanalysis data on a rectangle over France from $1.5^{\circ}$ to $4^{\circ}$ longitude and $47.75^{\circ}$ to $49.25^{\circ}$ latitude (basically centred in Paris). Fitting the GEV to annual maxima of these data, the mean estimates computed on the 77 grid points are $\eta=25.72$ m.s$^{-1}$, $\tau = 2.50$ m.s$^{-1}$ and $\xi=-0.14$. Due to these various results, we perform our study with the values $\eta=30$, $\tau=3$ and $\xi=-0.2$ for simplification; the results described below remain qualitatively the same with different values of these parameters. 
}

\Tb{
The integral appearing in the expression of $g_{\beta_1, \beta_2}^{(\mathrm{s})}$ (see \eqref{Eq_Def_g_beta1_beta2}) has no closed form and therefore a numerical approximation is required. For this purpose, we use adaptive quadrature with a relative accuracy of $3 \times 10^{-7}$. Furthermore, we choose the variogram \eqref{Eq_Power_Variogram} with $\kappa=1$ and $\psi=0.5, 1, 1.5, 2$. A value of the smoothness parameter $\psi$ between $0.25$ and $1.1$ seems reasonable for wind speed maxima; e.g, \cite{ribatet2013spatial} obtained $0.24$ $(0.02)$ and on similar data, \cite{einmahl2016m} found $0.40$ $(0.02)$, whereas we get $1.06$ $(0.07)$ on the region described above and $0.81$ $(0.06)$ over a region covering the Ruhr in Germany (the numbers inside the brackets denote the standard deviation). The difference may come from the fact that reanalysis tend to be smoother than in situ observations. Choosing a range $\kappa=1$ does not induce any loss of generality in our study as, should $\kappa$ differ from $1$, the appropriate plots would be the same as below with the values on the x-axis multiplied by $\kappa$. In practical situations, the true value of $\kappa$ should of course be taken. From \eqref{Eq_Variogram_Smith_Field}, we know that the case $\psi=2$ corresponds to the Smith field with $\Sigma=I_2$. Finally, the cases $\psi=1$ and $\psi=1.5$ are intermediate between the two previous settings. In accordance with the discussion above, Figure \ref{Plot_Surface_Corr_4panels} shows that $\mathcal{D}_{\beta, \eta, \tau, \xi}$ decreases from $1$ to $0$ as the Euclidean distance increases, and this at a higher rate for larger values of $\psi$. Especially, the decrease is faster for the Smith field than for all Brown--Resnick fields having $\psi<2$, and if the true value of $\psi$ is close to $0.5$ or even $1$, using the Smith model leads to a serious underestimation of the dependence between damage. The minimum Euclidean distance required for $\mathcal{D}_{\beta, \eta, \tau, \xi}$ to be lower than $0.01$ is about $1000$ for $\psi=0.5$ (not shown), instead of around $6$ for $\psi=2$. Moreover, interestingly, for a given Euclidean distance, $\mathcal{D}_{\beta, \eta, \tau, \xi}$ slightly increases in a concave way with the power $\beta$ but is basically constant, suggesting that our dependence measure is only faintly sensitive to the value of $\beta$. On top of being potentially insightful for the understanding of max-stable fields, this finding is valuable for actuarial practice as it means that the correlation between damage due to extreme wind speeds is basically the same whatever the value of the power. It also makes sense to consider the extension of \eqref{Eq_DepMeas} where the relevant damage power is $\beta_1$ at $\Mb{x}_1$ and $\beta_2$ at $\Mb{x}_2$ with $\beta_1$ and  $\beta_2$ not necessarily equal, as in Theorem \ref{Th_Cov_Maxstab_Real_Marg}. The behaviour is fairly similar to what we just described (not shown). Especially, for a given distance, the values are not highly sensitive to the combination $(\beta_1, \beta_2)$. For a fixed $\beta_1$, the correlation first increases and then decreases with $\beta_2$, but the value of $\beta_2$ achieving the maximum increases to $12$ as $\beta_1$ increases to $12$. Moreover, the higher the values of $\beta_1$, the higher the sensitivity with respect to $\beta_2$. The evolution of the covariance with respect to the distance is similar as for the correlation, but the sensitivity with $\beta$ (or the combination $(\beta_1, \beta_2)$ in the extension just mentioned) is much larger, owing to the non-normalized nature of covariance. 
\begin{figure}[h!]
\center
\includegraphics[scale=0.85]{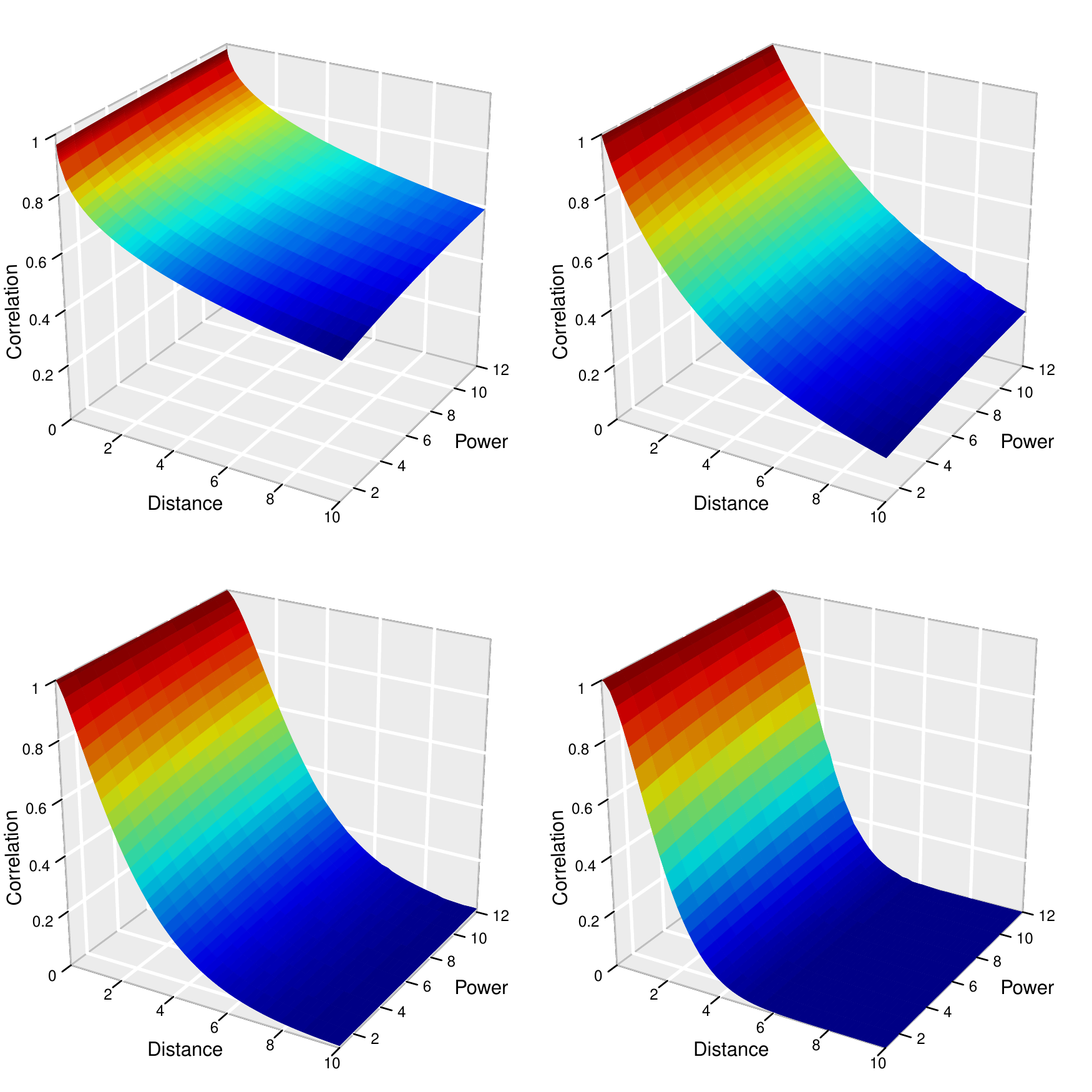}
\caption{Evolution of $\mathcal{D}_{\beta, 30, 3, -0.2}(\gamma_W(\Mb{x}_2-\Mb{x}_1))$ with respect to the distance $\| \Mb{x}_2-\Mb{x}_1 \|$ and the power $\beta$. The variogram of the Brown--Resnick field is $\gamma_W(\Mb{x})= \| \Mb{x} \|^{\psi}$, $\Mb{x} \in \mathbb{R}^2$; top left, top right, bottom left and bottom right panels correspond to $\psi=0.5, 1, 1.5$ and $2$, respectively.}
\label{Plot_Surface_Corr_4panels}
\end{figure}
}

\section{\Tb{Spatial risk measures induced by powers of max-stable fields and applications to wind extremes}} 
\label{Sec_Examples_SRM_Power_Damage_Function}

In this section, we study some examples of spatial risk measures induced by the cost field
\Beq
\label{Eq_Specific_Cost_Process}
\left \{ C(\Mb{x}) \right \}_{\Mb{x} \in \mathbb{R}^2}=\left \{ Z(\Mb{x})^{\beta} \right \}_{\Mb{x} \in \mathbb{R}^2},
\Eeq
where \Tb{$\{ Z(\Mb{x}) \}_{\Mb{x} \in \mathbb{R}^2}$ is a max-stable random field belonging to $\mathcal{C}$ with GEV parameters $\eta \in \mathbb{R}$, $\tau>0$, $\xi \neq 0$ and the power $\beta \in \mathbb{N}_*$ satisfies $\beta \xi < 1/2$ or $\beta \xi <1$. Most often, $Z$ will be the \Tb{Brown--Resnick} random field. As before, we keep a tight connection to concrete applications to the risk of losses due to extreme wind speeds.}

\Tb{We shall use the following lemma.}
\begin{Lem}
\label{Lemma_CP_Locally_Integrable_Sample_Paths}
Let $\{ Z(\Mb{x}) \}_{\Mb{x} \in \mathbb{R}^2}$ be a \Tb{measurable max-stable random field with GEV parameters $\eta \in \mathbb{R}$, $ \tau>0$ and $\xi \neq 0$. Let $\beta \in \mathbb{N}_*$ such that $\beta \xi < 1$}. Then, the random field $Z^{\beta}$ belongs to $\mathcal{C}$.
\end{Lem}
\begin{proof}
The field $Z^{\beta}$ is obviously measurable. \Tb{Furthermore, as $Z$ has identical univariate marginal distributions,
the function $\Mb{x} \mapsto \mathbb{E} [ |Z(\Mb{x})^{\beta} | ]$}
is constant and hence locally integrable. Therefore, Proposition 1 in \cite{koch2019SpatialRiskAxioms} yields that $Z^{\beta}$ has a.s. locally integrable sample paths.
\end{proof}
\Tb{Let $\{ Z(\Mb{x}) \}_{\Mb{x} \in \mathbb{R}^2}$ and $\beta$ be as in Lemma \ref{Lemma_CP_Locally_Integrable_Sample_Paths}. We consider the field $\{ C(\Mb{x}) \}_{\Mb{x} \in \mathbb{R}^2}=\left \{ Z(\Mb{x})^{\beta} \right \}_{\Mb{x} \in \mathbb{R}^2}$ and the spatial risk measure associated with the expectation $\mathcal{R}_1(A, C)=\mathbb{E} \left[ L_N(A, C) \right], A \in \mathcal{A}$. It is clear that $C$ is measurable. In addition, since it has identical univariate marginal distributions and $\beta \xi <1$, it has a constant expectation and, for any $\Mb{x} \in \mathbb{R}^2$, $\mathbb{E}[|C(\Mb{x})|]=\mathbb{E}[|C(\Mb{0})|] < \infty$. 
Consequently, Theorem 3 in \cite{koch2019SpatialRiskAxioms} gives that, for all $A \in \mathcal{A}$, 
$\mathcal{R}_1(A,C)= \mathbb{E} \left[ C(\Mb{0}) \right]$ and that $\mathcal{R}_1(\cdot, C)$ satisfies the axioms of spatial invariance under translation and spatial sub-additivity. Provided that $\mathbb{E}[C(\Mb{0})] \neq 0$, it also yields that $\mathcal{R}_1(\cdot, C)$ satisfies the axiom of asymptotic spatial homogeneity of order $0$ with $K_1(A,C)=0$ and $K_2(A,C)=\mathbb{E} \left[ C(\Mb{0}) \right]$, $A \in \mathcal{A}_c$.
Using \eqref{Eq_Link_Maxstb_Simple_Maxstab}, the binomial theorem and Lemma \ref{Chapriskmeasure_Lem_Frechet_Margins}, we obtain
\Beq
\label{Ex_Exp_Z_beta}
\mathbb{E} \left[ C(\Mb{0}) \right]=\sum_{k=0}^{\beta} {\beta \choose k} \left( \eta-\frac{\tau}{\xi} \right)^{k} \left( \frac{\tau}{\xi} \right)^{\beta-k} \Gamma(1-[\beta-k]\xi).
\Eeq  
}

\Tb{Spatial risk measures associated with the expectation are not of great interest since, by Fubini's Theorem, they do not account for the spatial dependence of the cost field $C$.} In the following, we study some spatial risk measures associated with variance in detail. Then, we provide a central limit-based approximation of the distribution of $L_N(\lambda A, C)$ for $A \in \mathcal{A}_c$ and $\lambda$ large enough. Finally, we analyze some spatial risk measures associated with VaR and ES. \Tb{For a discussion about the respective advantages and drawbacks of variance, VaR and ES as classical risk measures, we refer the reader to \citet[][Section 3]{koch2019SpatialRiskAxioms}.}

\subsection{Spatial risk measures associated with variance}
\label{Subsec_Variance}

We focus on $\mathcal{R}_2(\cdot, C)=\mbox{Var}\left( L_N(\cdot, C) \right)$, provided it is finite, where $C$ is given in \eqref{Eq_Specific_Cost_Process}. \Tb{We first study in detail} the function $\lambda \mapsto \mathcal{R}_2(\lambda A, C)$ for specific regions $A \in \mathcal{A}$. Among others, we derive useful expressions for $\mathcal{R}_2(\lambda A, C)$, $\lambda>0$, \Tb{that may be} of practical relevance for the insurance industry \Tb{and} will allow us to prove the axiom of spatial sub-additivity in specific configurations. Second, we provide conditions on the field $Z$ such that $\mathcal{R}_2(\cdot, C)$ satisfies (at least) part of the axioms \Tb{presented in Section \ref{Subsec_Spat_Risk_Axioms}}.

\subsubsection{Study of $\mathcal{R}_2(\lambda A, C), \lambda >0$}
\label{Subsubsec_R2_lambdaA}

Let $\left \{ Z(\Mb{x}) \right \}_{\Mb{x} \in \mathbb{R}^2}$ be a \Tb{measurable max-stable random field with GEV parameters $\eta \in \mathbb{R}$, $\tau>0$ and $\xi \neq 0$. Moreover, let $\beta \in \mathbb{N}_*$ such that $\beta \xi < 1/2$ and $\{ C(\Mb{x}) \}_{\Mb{x} \in \mathbb{R}^2}=\{ Z(\Mb{x})^{\beta} \}_{\Mb{x} \in \mathbb{R}^2}$.} Lemma \ref{Lemma_CP_Locally_Integrable_Sample_Paths} yields that $C \in \mathcal{C}$ \Tb{and it easily follows from Lemma \ref{Chapriskmeasure_Lem_Frechet_Margins} that, for all $A \in \mathcal{A}$, $\sup_{\Mb{x} \in A} \{ \mathbb{E} \left[C(\Mb{x})^2\right] \} < \infty$}. Thus, using Theorem 4 in \cite{koch2019SpatialRiskAxioms}, we obtain that, for all $A \in \mathcal{A}$ and $\lambda>0$,
\Beq
\label{Eq_R2_General}
\mathcal{R}_2(\lambda A, C)= \frac{1}{\lambda^4 [\nu(A)]^2} \int_{\lambda A}  \int_{\lambda A} \mathrm{Cov}\left( Z(\Mb{x})^{\beta}, Z(\Mb{y})^{\beta} \right) \  \nu(\mathrm{d}\Mb{x})\  \nu(\mathrm{d}\Mb{y}).
\Eeq
Hence, taking advantage of the expression of $\mathrm{Cov} \left( Z(\Mb{x})^{\beta},  Z(\Mb{y})^{\beta} \right)$ obtained in  \eqref{Eq_Cov_Maxstab_Real_Marg_Eq_Coeff}, we deduce the expression of $\mathcal{R}_2(\lambda A, C)$ in the case where \Tb{$Z$ is the Brown--Resnick random field with an isotropic variogram and} when the region $A$ is a disk or a square. In the whole paper, disk and square refer to a closed disk with positive radius and a (closed) square with positive side, respectively. The corresponding results are given in the next \Tb{theorem}. 

\Tb{
\begin{Th}
\label{Th_R2_Lambda_Finite_BR}
Let $\{ Z(\Mb{x}) \}_{\Mb{x} \in \mathbb{R}^2}$ be a measurable Brown--Resnick random field associated with an isotropic variogram $\gamma_W$ whose corresponding univariate function is $\gamma_{W, \mathrm{u}}$ and with GEV parameters $\eta \in \mathbb{R}$, $\tau>0$ and $\xi \neq 0$. 
Let $\beta \in \mathbb{N}_*$ such that $\beta \xi < 1/2$ and $\{ C(\Mb{x}) \}_{\Mb{x} \in \mathbb{R}^2}= \{ Z(\Mb{x})^{\beta} \}_{\Mb{x} \in \mathbb{R}^2}$. Then:
\Ben
\item Let $A$ be a disk with radius $R$. For all $\lambda >0$, we have
\begin{align}
\label{Eq_Expression_R2_Disk}
\mathcal{R}_2(\lambda A, C) & = - \sum_{k_1=0}^{\beta} \sum_{k_2=0}^{\beta} B_{k_1, k_2, \beta, \eta, \tau, \xi} \ \Gamma (1-[\beta-k_1]\xi ) \Gamma(1-[\beta-k_2]\xi) \nonumber \\
& \quad \ + \int_{h=0}^{2R} f_d(h,R) \ g_{\beta, \eta, \tau, \xi} \left( \sqrt{\gamma_{W, \mathrm{u}}(\lambda h)} \right)\  \nu(\mathrm{d}h),
\end{align}
where $f_d$ is the density of the Euclidean distance between two points independently and uniformly distributed on $A$, given, for $h \in [0, 2R]$, by
$$
f_d(h,R) = \frac{2h}{R^2} \left( \frac{2}{\pi} \arccos \left( \frac{h}{2 R} \right) - \frac{h}{\pi R} \sqrt{1-\frac{h^2}{4 R^2}}   \right).
$$
\item Let $A$ be a square with side $R$. For all $\lambda >0$, we have
\begin{align*}
\mathcal{R}_2(\lambda A, C) & = - \sum_{k_1=0}^{\beta} \sum_{k_2=0}^{\beta} B_{k_1, k_2, \beta, \eta, \tau, \xi} \ \Gamma (1-[\beta-k_1]\xi ) \Gamma(1-[\beta-k_2]\xi) \\
& \quad \ + \int_{h=0}^{R \sqrt{2}} f_s(h,R) \ g_{\beta, \eta, \tau, \xi} \left( \sqrt{\gamma_{W, \mathrm{u}}(\lambda h)} \right)\  \nu(\mathrm{d}h),
\end{align*}
where $f_s$ is the density of the Euclidean distance between two points independently and uniformly distributed on $A$, written as
$$
f_s(h,R)=
\left \{
\begin{array}{ll}
\frac{2 \pi h}{R^2} - \frac{8 h^2}{R^3} + \frac{2 h^3}{R^4}, & \quad h \in [0,R], \\
\Big ( -2-b+ 3 \sqrt{b-1} + \frac{b+1}{\sqrt{b-1}} +2 \arcsin \left( \frac{2-b}{b} \right) - \frac{4}{b \sqrt{1-(2-b)^2/b^2}}\Big) \frac{2h}{R^2}, & \quad h \in  [R, R \sqrt{2}],
\end{array}
\right.
$$
with $b=h^2/R^2$.
\Een
\end{Th}
}
\begin{proof}
\Tb{Using \eqref{Eq_Cov_Maxstab_Real_Marg_Eq_Coeff}, we obtain, for $\Mb{x}_1, \Mb{x}_2 \in \mathbb{R}^2$, 
\begin{align*}
& \quad \ \mbox{Cov} \left( Z(\Mb{x}_1)^{\beta}, Z(\Mb{x}_2)^{\beta} \right)
\\& = g_{\beta, \eta, \tau, \xi} \left( \sqrt{\gamma_{W, \mathrm{u}}(\| \Mb{x}_2-\Mb{x}_1 \|)} \right) - \sum_{k_1=0}^{\beta} \sum_{k_2=0}^{\beta} B_{k_1, k_2, \beta, \eta, \tau, \xi} \ \Gamma(1-[\beta-k_1]\xi) \Gamma(1-[\beta-k_2]\xi).
\end{align*}}
The results follow from \eqref{Eq_R2_General} and similar arguments as in the proof of Corollary 1 in \cite{koch2017spatial}.
\end{proof}

\Tb{Let $Z_{\xi}$ and $Z_0$ as in Proposition \ref{Prop_Continuity_Cov_xi_0}, and $\{ C(\Mb{x}) \}_{\Mb{x} \in \mathbb{R}^2}=\{ Z_{\xi}(\Mb{x})^{\beta} \}_{\Mb{x} \in \mathbb{R}^2}$. Using Cauchy--Schwarz inequality, we can easily apply the Lebesgue dominated convergence theorem to show that
$$\lim_{\xi \to 0} \mathcal{R}_2(\lambda A, C)= \frac{1}{\lambda^4 [\nu(A)]^2} \int_{\lambda A}  \int_{\lambda A} \mathrm{Cov}\left( Z_0(\Mb{x})^{\beta}, Z_0(\Mb{y})^{\beta} \right) \  \nu(\mathrm{d}\Mb{x})\  \nu(\mathrm{d}\Mb{y}).$$
Therefore, again, the case $\xi=0$ is recovered by letting $\xi$ tend to $0$.}

The following theorem concerns the limit of $\mathcal{R}_2 \left(\lambda A, Z^{\beta} \right)$ as $\lambda \to \infty$.
\Tb{
\begin{Th}
\label{Th_Limiting_Risk_Measure}
Let $Z$, $\beta$ and $C$ be as in Theorem \ref{Th_R2_Lambda_Finite_BR} and assume moreover that $\gamma_{W, \mathrm{u}}$ is measurable and satisfies
\Beq
\label{Eq_Limit_Gamma_u}
\lim_{h \to \infty} \gamma_{W, \mathrm{u}}(h)= \infty.
\Eeq
Then, for all $A$ being a disk with radius $R$ or a square with side $R$, $\lim_{\lambda \to \infty} \mathcal{R}_2 (\lambda A, C)=0$. 
\end{Th}
\begin{proof}
We show the result when $A$ is a disk; the arguments are the same in the case of the square. 
Since $f_d$ is a continuous function of $h$ on $[0, 2R]$, it is bounded. By Propositions~\ref{Prop_Decrease_gtilde}--\ref{Prop_Lim_gtilde_infty}, $g_{\beta, \eta, \tau, \xi}$ is also bounded and consequently there exists $U>0$ such that, for all $h \geq 0$ and $\lambda>0$, $| f_d(h,R) \ g_{\beta, \eta, \tau, \xi} ( \sqrt{\gamma_{W, \mathrm{u}}(\lambda h)} ) | \leq U$. In addition, $f_d$ and $g_{\beta, \eta, \tau, \xi}$ are continuous and thus measurable, which yields by measurability of $\gamma_{W, \mathrm{u}}$ that, for all $\lambda>0$, the function $h \mapsto f_d(h,R) \ g_{\beta, \eta, \tau, \xi} ( \sqrt{\gamma_{W, \mathrm{u}}(\lambda h)} )$ is measurable. Thus, by Lebesgue's dominated convergence theorem, the fact that $f_d$ is a density and \eqref{Eq_Limit_Gamma_u}, 
\begin{align}
\label{Eq_Lim_Int_Lambdah}
\lim_{\lambda \to \infty} \int_{h=0}^{2R} f_d(h,R) \ g_{\beta, \eta, \tau, \xi} \left( \sqrt{\gamma_{W, \mathrm{u}}(\lambda h)} \right)\  \nu(\mathrm{d}h) &= \int_{h=0}^{2R} f_d(h,R) \ \lim_{\lambda \to \infty}  g_{\beta, \eta, \tau, \xi} \left( \sqrt{\gamma_{W, \mathrm{u}}(\lambda h)} \right)\  \nu(\mathrm{d}h) \nonumber
\\& = \lim_{\lambda \to \infty}  g_{\beta, \eta, \tau, \xi}(\lambda).
\end{align}
Finally, combining \eqref{Eq_Lim_gtilde_infty}, \eqref{Eq_Expression_R2_Disk} and \eqref{Eq_Lim_Int_Lambdah}, we obtain
$$ \lim_{\lambda \to \infty} \mathcal{R}_2(\lambda A, C)= - \sum_{k_1=0}^{\beta} \sum_{k_2=0}^{\beta} B_{k_1, k_2, \beta, \eta, \tau, \xi} \ \Gamma(1-[\beta-k_1]\xi) \Gamma(1-[\beta-k_2]\xi) + \lim_{\lambda \to \infty} g_{\beta, \eta, \tau, \xi}(\lambda)=0.$$
\end{proof}
}

\Tb{Let $Z$ and $C$ be as in Theorem~\ref{Th_R2_Lambda_Finite_BR}. If the function $h \mapsto \gamma_{W, \mathrm{u}}(h)$ is strictly increasing (respectively increasing), then Proposition \ref{Prop_Decrease_gtilde} implies that, for all $h>0$, the function $\lambda \mapsto g_{\beta, \eta, \tau, \xi} ( \sqrt{\gamma_{W, \mathrm{u}}(\lambda h)} )$ is strictly decreasing (respectively decreasing). Therefore, it follows from Theorem \ref{Th_R2_Lambda_Finite_BR} that $\lambda \mapsto \mathcal{R}_2(\lambda A, C)$ is strictly decreasing (respectively decreasing) for $A$ being a disk or a square; there is thus spatial diversification. Furthermore, if $\gamma_{W, \mathrm{u}}$ is measurable and satisfies \eqref{Eq_Limit_Gamma_u}, Theorem~\ref{Th_Limiting_Risk_Measure} entails that this spatial diversification is total, which has to be understood in the sense that $\lim_{\lambda \to \infty} \mathcal{R}_2(\lambda A, C)=0$. Theorem~\ref{Th_R2_Lambda_Finite_BR} is of interest for the insurance industry as it allows a company to compute the value of $\lambda$ such that $\mathcal{R}_2(\lambda A, C)$ equals a wanted low variance level. In other words, it enables one to find out the characteristic dimension of a geographical area needed to reach a specified low variance for the loss per surface unit.}

\Tb{Below, we study how $\mathcal{R}_2(\lambda A, C)$ evolves with respect to $\lambda$ under the assumptions of Theorem \ref{Th_R2_Lambda_Finite_BR} for various values of $\beta$. The integrals involved have no closed form and thus, as above, we use adaptive quadrature with $3 \times 10^{-7}$ as relative accuracy. We set without loss of generality $R=1$ and, as in Section \ref{Sec_Dependence_Power_Maxstable}, we choose $\eta=30$, $\tau=3$, $\xi=-0.2$, and consider the variogram 
$\gamma_W(\Mb{x})= \| \Mb{x} \|^{\psi}$, $\Mb{x} \in \mathbb{R}^2$, with $\psi=0.5, 1, 1.5, 2$; similar results are obtained with different marginal parameters. Figure \ref{Plot_R2_1} displays a rapid decrease to $0$ in the case $\psi=2$ (blue curve) which corresponds to the Smith random field; the decrease is somewhat slower in the case of the square. This behaviour is similar to the one observed when the cost field is the indicator function of the Smith random field exceeding a given threshold; see \cite{koch2017spatial}, Figure 1. Figure \ref{Plot_R2_1} also shows that, for a fixed $\beta$, the decrease to $0$ becomes much slower when $\psi$ decreases, which was theoretically expected since the rate of convergence to $0$ raises as the rate of divergence of the variogram to infinity increases. Again, an insurance company that would model wind speed extremes using the Smith model would substantially underestimate its risk if the true value of $\psi$ is less than $1$. The curves are very similar for other values of $\beta$, but for a given $\lambda$ the value of $\mathcal{R}_2(\lambda A, C)$ increases basically exponentially with $\beta$; the plot of $\log (\mathcal{R}_2(\lambda A, C))$ with respect to $\beta$ is essentially linear (not shown). This feature is comprehensible in view of the expression of $g_{\beta, \eta, \tau, \xi}$ and the evolution of the function $g_{\beta}^{(\mathrm{s})}$ with respect to $\beta$ (see Figure \ref{Plot_Persp_gbeta} in Appendix \ref{Sec_Appendix_Maxstable}). Hence, although the value of $\beta$ has little impact on the spatial dependence measure $\mathcal{D}_{\beta, \eta, \tau, \xi}$, it strongly influences the values of $\mathcal{R}_2(\lambda A, C)$, $\lambda>0$.}
\Tb{As in Section \ref{Sec_Dependence_Power_Maxstable}, choosing a range $\kappa$ different from $1$ would not affect our conclusions. It would just modify the values on the x-axis: the larger the range, the slower the spatial diversification.}
\begin{figure}[h!]
\center
\includegraphics[scale=0.7]{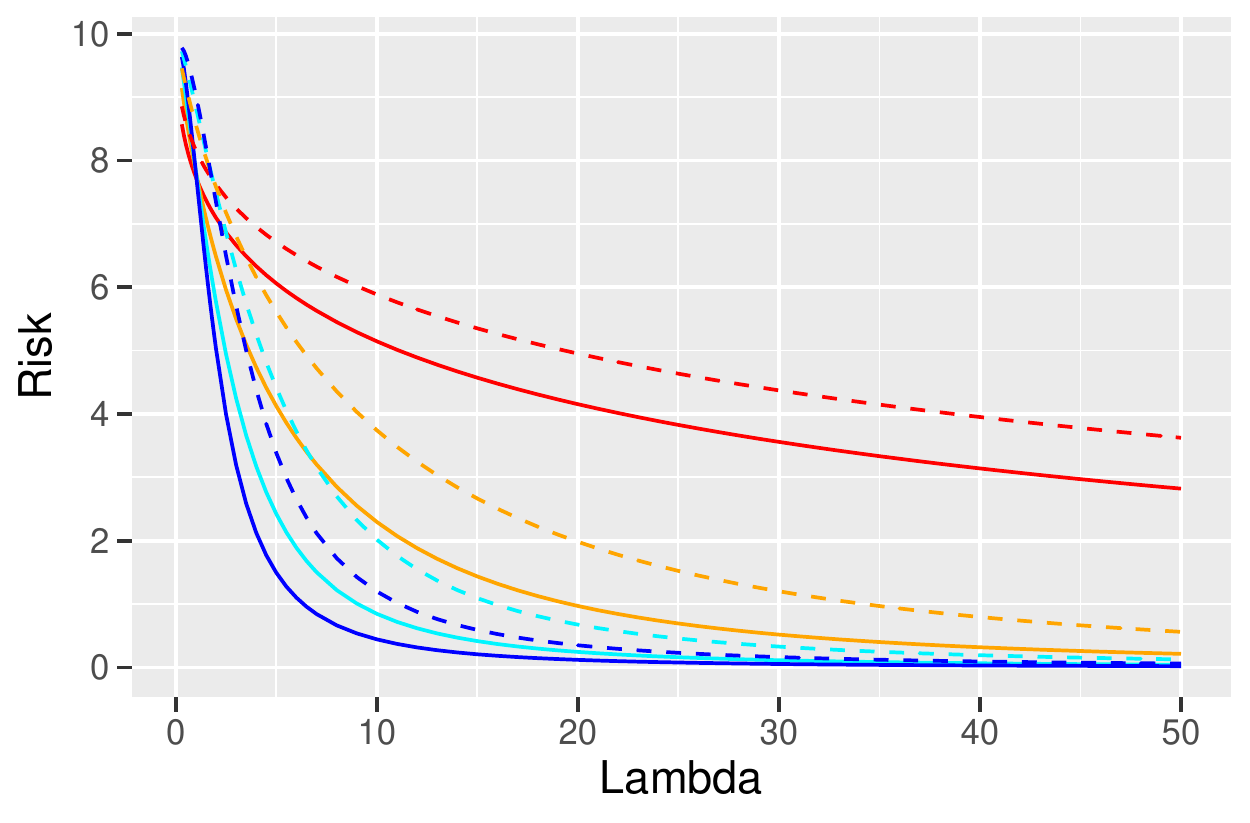}
\caption{The solid (respectively dashed) lines depict the evolution of $\mathcal{R}_2(\lambda A, C)$ with respect to $\lambda$ in the case $\beta=1$, where $A$ is a disk (respectively a square) with radius (respectively side) $R=1$. The variogram of the Brown--Resnick field is $\gamma_W(\Mb{x})= \| \Mb{x} \|^{\psi}$, $\Mb{x} \in \mathbb{R}^2$; red, orange, turquoise and blue correspond to $\psi=0.5$, $1$, $1.5$ and $2$, respectively.}
\label{Plot_R2_1}
\end{figure}

\subsubsection{Axioms}
\label{Subsubsec_Axioms}

\Tb{
We start with a preliminary result. Let $\mathcal{B}(\mathbb{R})$ and $\mathcal{B}((0, \infty))$ denote the Borel $\sigma$-fields on $\mathbb{R}$ and $(0, \infty)$, respectively. 
\begin{Lem}
\label{Lem_Moment_2plusdelta}
Let $\{ Z^{(\mathrm{s})}(\Mb{x}) \}_{\Mb{x} \in \mathbb{R}^2}$ be a simple max-stable random field. Let $\eta \in \mathbb{R}$, $\tau>0$, $\xi \in \mathbb{R}$ and $\beta \in \mathbb{N}_*$.
The function defined by
\Beq
\label{Eq_D_beta_eta_tau_xi}
D_{\beta, \eta, \tau, \xi}(z)=
\left\{
\begin{array}{ll}
\left[ \left( \eta-\tau/\xi \right) + \tau z^{\xi}/\xi \right]^{\beta}, & \quad \xi \neq 0, \\
\left[ \eta + \tau \log(z) \right]^{\beta}, & \quad \xi = 0,
\end{array}
\qquad z >0, \right.
\Eeq
is measurable from $((0, \infty), \mathcal{B}((0,\infty)))$ to $(\mathbb{R},\mathcal{B}^1)$ and strictly increasing. Moreover, if $\beta \xi <1/2$, then $\mathbb{E} [ | D_{\beta, \eta, \tau, \xi}(Z^{(\mathrm{s})}(\Mb{0})) |^{2+\delta} ]<~\infty$ for any $\delta$ such that $0 < \delta < 1/(\xi \beta)-2$.
\end{Lem}
\begin{proof}
The fact that $D$ is measurable and strictly increasing is obvious. Denoting $Z= [D_{\beta, \eta, \tau, \xi}(Z^{(\mathrm{s})}(\Mb{0}))]^{1/\beta}$, we have, for $\delta>0$,
$$\mathbb{E} \left[ \left| D_{\beta, \eta, \tau, \xi}(Z^{(\mathrm{s})}(\Mb{0})) \right|^{2+\delta} \right]=\mathbb{E} \left[ \left|Z^{\beta} \right|^{2+\delta}\right] = \mathbb{E}\left[\left |Z \right|^{\beta(2+\delta)}\right],$$
which is finite (see the proof of Proposition \ref{Prop_Continuity_Cov_xi_0}) provided $\beta(2+\delta) \xi < 1$ as $Z$ follows the GEV with parameters $\eta$, $\tau$ and $\xi$. The latter inequality is satisfied for any positive $\delta$ such that $\delta < 1/(\xi \beta)-2$.
\end{proof}
}

\Tb{
We shall use the following remark.
\begin{Rq}
It is easily seen that Theorem 3 in \cite{koch2017TCL} also holds true if $Z$ is a simple Brown--Resnick field with variogram $\gamma_W(\Mb{x})=m  \| \Mb{x} \|_{\Sigma}^{\psi}$ (and not only $m  \| \Mb{x} \|^{\psi}$), where $m >0$, $\psi \in (0,2]$ and $\Sigma$ is any symmetric positive-definite matrix; this implies that Theorem 8 and Corollary 4 in \cite{koch2019SpatialRiskAxioms} are also true in this more general setting. In the following, when we refer to Theorem 8 and Corollary 4 in \cite{koch2019SpatialRiskAxioms}, we refer to this extended version.
\end{Rq}
}

The following theorem provides conditions on the field $Z$ such that $\mathcal{R}_2(\cdot, C)$ satisfies (at least) part of the axioms \Tb{in Section \ref{Subsec_Spat_Risk_Axioms}}.
\Tb{
\begin{Th}
\label{Th_Axioms_R2}
\Ben
\item Let $\{Z(\Mb{x})\}_{\Mb{x} \in \mathbb{R}^2}$ be a stationary and measurable max-stable random field with GEV parameters $\eta \in \mathbb{R}$, $\tau>0$ and $\xi \in \mathbb{R}$. Let $\beta \in \mathbb{N}_*$ such that $\beta \xi < 1/2$. We introduce the cost field $\{ C(\Mb{x}) \}_{\Mb{x} \in \mathbb{R}^2}= \{ Z(\Mb{x})^{\beta} \}_{\Mb{x} \in \mathbb{R}^2}$. Then, the spatial risk measure induced by $C$ $\mathcal{R}_2(\cdot, C)$ satisfies the axiom of spatial invariance under translation. In particular, this is true for the tube random field as well as the measurable Schlather and Brown--Resnick fields (which includes the Smith field).
\item Let $\{Z(\Mb{x})\}_{\Mb{x} \in \mathbb{R}^2}$ be a measurable Brown--Resnick random field associated with an isotropic variogram $\gamma_W$ such that $\gamma_{W, \mathrm{u}}$ is increasing and with GEV parameters $\eta \in \mathbb{R}$, $\tau>0$ and $\xi \in \mathbb{R}$. Let $\beta \in \mathbb{N}_*$ such that $\beta \xi < 1/2$ and $\{ C(\Mb{x}) \}_{\Mb{x} \in \mathbb{R}^2}= \{ Z(\Mb{x})^{\beta} \}_{\Mb{x} \in \mathbb{R}^2}$.
Then $\mathcal{R}_2(\cdot, C)$ satisfies the axiom of spatial sub-additivity when the two regions are both a disk or a square. The axiom is satisfied with strict inequality if $\gamma_{W, \mathrm{u}}$ is strictly increasing, as in the case of the isotropic Smith field.
\item Let $\{Z(\Mb{x})\}_{\Mb{x} \in \mathbb{R}^2}$ be the Brown--Resnick random field associated with the variogram $\gamma_W(\Mb{x})=m \| \Mb{x} \|_{\Sigma}^{\psi}$, $\Mb{x} \in \mathbb{R}^2$, where $m >0$, $\psi \in (0,2]$ and $\Sigma$ is any symmetric positive-definite matrix. Assume that the GEV parameters of $Z$ are $\eta \in \mathbb{R}$, $\tau>0$ and $\xi \in \mathbb{R}$.
Let $\beta \in \mathbb{N}_*$ such that $\beta \xi <1/2$ and $\{ C(\Mb{x}) \}_{\Mb{x} \in \mathbb{R}^2}=\{ Z(\Mb{x})^{\beta} \}_{\Mb{x} \in \mathbb{R}^2}$. Then $\mathcal{R}_2(\cdot, C)$ satisfies the axiom of asymptotic spatial homogeneity of order $-2$, with
$$K_1(A,C)=0 \quad \mbox{ and } \quad K_2(A,C)= \frac{1}{\nu(A)} \int_{\mathbb{R}^2} \mathrm{Cov} \left( Z(\Mb{0})^{\beta}, Z(\Mb{x})^{\beta}\right) \nu(\mathrm{d}\Mb{x}), \quad A \in \mathcal{A}_c,$$
where the expression of $\mathrm{Cov}( Z(\Mb{0})^{\beta}, Z(\Mb{x})^{\beta})$, $\Mb{x} \in \mathbb{R}^2$, is given by \eqref{Eq_Cov_Maxstab_Real_Marg_Eq_Coeff}.
\Een
\end{Th}
}
\begin{proof}
\Tb{
1. By Lemma \ref{Lemma_CP_Locally_Integrable_Sample_Paths}, $C \in \mathcal{C}$. Moreover, since $Z$ is stationary, it is also the case for $C$. Additionally, as mentioned immediately before \eqref{Eq_R2_General}, the induced spatial risk measure $\mathcal{R}_2(\cdot, C)$ is well-defined. Finally, Var is a law-invariant classical risk measure. Hence, Theorem 5, Point 1, in \cite{koch2019SpatialRiskAxioms} gives the first result. As an instance of moving maxima random field, the tube model is stationary and measurable. Hence, the specific fields mentioned satisfy the assumptions, concluding the proof.
}

\Tb{
2. First, $\mathcal{R}_2(\cdot, C)$ is invariant under translation by Point 1. Moreover, for $A$ being a disk or a square, we have seen in Section \ref{Subsubsec_R2_lambdaA} that $\lambda \mapsto \mathcal{R}_2(\lambda A, C)$ is decreasing (respectively strictly decreasing) if $\gamma_{W, \mathrm{u}}$ is increasing (respectively strictly increasing). Therefore, the result follows by the same reasoning as in the proof of Theorem 3, Point 2, in \cite{koch2017spatial}.
}

\Tb{
3. It follows from \eqref{Eq_Link_Maxstb_Simple_Maxstab} and \eqref{Eq_D_beta_eta_tau_xi} that
$\{C(\Mb{x})\}_{\Mb{x} \in \mathbb{R}^2} = \{D_{\beta, \eta, \tau, \xi}(Z^{(\mathrm{s})}(\Mb{x}))\}_{\Mb{x} \in \mathbb{R}^2}$, where $\{ Z^{(\mathrm{s})}(\Mb{x}) \}_{\Mb{x} \in \mathbb{R}^2}$ is simple max-stable. Lemma \ref{Lem_Moment_2plusdelta} gives that $D_{\beta, \eta, \tau, \xi}$ satisfies the assumptions of Corollary 4 in \cite{koch2019SpatialRiskAxioms}. Thus, the result directly follows from Corollary 4 in \cite{koch2019SpatialRiskAxioms}. 
} 
\end{proof}
\Tb{
In the case of damage caused by extreme wind speeds, the condition $\beta \xi<1/2$ is generally satisfied for any $\beta \in \mathbb{N}_*$ since, most often, $\xi<0$. When $\xi>0$, the corresponding value is typically very close to $0$ and this condition is satisfied for low to moderate powers. Hence, in concrete applications, the results of Theorem \ref{Th_Axioms_R2} usually hold true.
\begin{Rq}
We think intuitively that Point 2 of Theorem \ref{Th_Axioms_R2} is also true for all $A \in \mathcal{A}_c$ or even $A \in \mathcal{A}$ and not only for disks and squares. However, it is difficult to prove it with the current argument based on the expressions of $\mathcal{R}_2(\lambda A, C)$ of Theorem \ref{Th_R2_Lambda_Finite_BR} as, for more complex geometric shapes of the region $A$, little is known on the
density of the distance between two points independently and uniformly distributed on $A$ \citep[e.g.,][Section 4.3.3]{moltchanov2012distance}. 
\end{Rq}
\begin{Rq}
\label{Rq_Extens_BR}
Let $\{ Z(\Mb{x}) \}_{\Mb{x} \in \mathbb{R}^2}$ have GEV parameters $\eta \in \mathbb{R}$, $\tau>0$ and $\xi \in \mathbb{R}$.
Furthermore, let $\beta \in \mathbb{N}_*$ such that $\beta \xi <1/2$. The result of Theorem \ref{Th_Axioms_R2}, Point 3, is also true if $Z$ is a sample-continuous Brown--Resnick random field associated with a variogram $\gamma_W$ which satisfies the slightly weaker condition
\Beq
\label{Eq_Condition_Extremal_Coefficient_BR}
\int_{\mathbb{R}^2} \left[ 2-\Phi \left( \sqrt{\gamma_W(\Mb{x})}/2 \right) \right]^{\delta/(2+\delta)} \ \nu(\mathrm{d} \Mb{x}) < \infty,
\Eeq
for some $\delta$ satisfying $0 < \delta < 1/(\xi \beta)-2$.
\end{Rq}
\begin{proof}
The field $Z$ is sample-continuous and thus measurable, which yields by Lemma \ref{Lemma_CP_Locally_Integrable_Sample_Paths} that $C \in \mathcal{C}$. By stationarity of the Brown--Resnick field, for all $\Mb{x}, \Mb{y} \in \mathbb{R}^2$, $\mathrm{Cov}(C(\Mb{x}), C(\Mb{y}))=\mathrm{Cov}(C(\Mb{0}), C(\Mb{x}-\Mb{y}))$. Now, it follows from \eqref{Eq_Link_Maxstb_Simple_Maxstab} and \eqref{Eq_D_beta_eta_tau_xi} that $\{C(\Mb{x})\}_{\Mb{x} \in \mathbb{R}^2} = \{D_{\beta, \eta, \tau, \xi}(Z^{(\mathrm{s})}(\Mb{x}))\}_{\Mb{x} \in \mathbb{R}^2}$, where $\{ Z^{(\mathrm{s})}(\Mb{x}) \}_{\Mb{x} \in \mathbb{R}^2}$ is a simple and sample-continuous max-stable field. Lemma \ref{Lem_Moment_2plusdelta} gives that $\mathbb{E} \left[ | C(\Mb{0}) |^{2+\delta} \right]< \infty$ and that $D_{\beta, \eta, \tau, \xi}$ satisfies the requirements for $F$ in Proposition 1 in \cite{koch2017TCL}. The latter yields $\int_{\mathbb{R}^2} \mathrm{Cov} \left( C(\Mb{0}), C(\Mb{x}) \right) \nu(\mathrm{d}\Mb{x}) >0$. Denoting by $\Theta$ the extremal coefficient of the Brown--Resnick field, \eqref{Eq_Condition_Extremal_Coefficient_BR} precisely means that $\int_{\mathbb{R}^2} [2-\Theta(\Mb{0}, \Mb{x})]^{\delta/(2+\delta)} \ \nu(\mathrm{d}\Mb{x}) < \infty$. Finally, the result follows from Theorem 6 in \cite{koch2019SpatialRiskAxioms}.
\end{proof}
}

\subsection{Central limit theorem and homothety}
\label{Subsec_CLT_Homothety} 

We first recall the concepts of Van Hove sequence and central limit theorem (CLT) for random fields \Tb{on $\mathbb{R}^d$}. For $\mathcal{V} \subset \mathbb{R}^d$ and $r>0$, we denote $\mathcal{V}^{+r}=\{ \Mb{x} \in \mathbb{R}^d: \mathrm{dist}(\Mb{x}, \mathcal{V}) \leq r \}$, where $\mathrm{dist}$ stands for the Euclidean distance. Additionally, we denote by $\partial \mathcal{V}$ the boundary of $\mathcal{V}$. A Van Hove sequence in $\mathbb{R}^d$ is a sequence $( \mathcal{V}_n )_{n \in \mathbb{N}}$ of bounded measurable subsets of $\mathbb{R}^d$ satisfying $\mathcal{V}_n \uparrow \mathbb{R}^d$, $\lim_{n \to \infty} \nu(\mathcal{V}_n)=\infty$, and $\lim_{n \to \infty} \nu((\partial \mathcal{V}_n)^{+r} )/\nu(\mathcal{V}_n) =0, \mbox{ for all } r>0$. We say that a random field $\{ C(\Mb{x}) \}_{\Mb{x} \in \mathbb{R}^d}$ such that, for all $\Mb{x} \in \mathbb{R}^d$, $\mathbb{E}\left[ C(\Mb{x})^2 \right]< \infty$, satisfies the CLT, if
$$
\int_{\mathbb{R}^d} | \mathrm{Cov}(C(\Mb{0}), C(\Mb{x})) | \  \nu(\mathrm{d}\Mb{x}) < \infty,
$$ 
and, for any Van Hove sequence $(\mathcal{V}_n)_{n \in \mathbb{N}}$ in $\mathbb{R}^d$,
$$
\frac{1}{\sqrt{\nu(\mathcal{V}_n)}} \int_{\mathcal{V}_n} (C(\Mb{x})-\mathbb{E}[C(\Mb{x})]) \  \nu(\mathrm{d}\Mb{x}) \overset{d}{\rightarrow} \mathcal{N} \left(0, \int_{\mathbb{R}^d} \mathrm{Cov}(C(\Mb{0}), C(\Mb{x})) \  \nu(\mathrm{d}\Mb{x}) \right), \quad \mbox{as } n\to\infty,
$$
where $\mathcal{N}(\mu, \sigma^2)$ denotes the normal distribution with expectation $\mu \in \mathbb{R}$ and variance $\sigma^2>0$.

Using results about CLT for functions of stationary max-stable random fields by \cite{koch2017TCL} and an outcome of \cite{koch2019SpatialRiskAxioms}, we obtain the following theorem.
\Tb{
\begin{Th} 
\label{TCL_Wind}
Let $Z$, $\beta$ and $C$ be as in Theorem \ref{Th_Axioms_R2}, Point 3.
We have, for all $A \in \mathcal{A}_c$, that
$$ \lambda  \left[ L_N(\lambda A, C) - \mu_{\beta, \eta, \tau, \xi} \right] \overset{d}{\to} \mathcal{N}\left( 0, \frac{1}{\nu(A)} \int_{\mathbb{R}^2} \mathrm{Cov} \left( Z(\Mb{0})^{\beta}, Z(\Mb{x})^{\beta} \right) \nu(\mathrm{d}\Mb{x}) \right), \mbox{ for } \lambda \to \infty,$$
where 
\Beq
\label{Eq_mu_beta_eta_tau_xi}
\mu_{\beta, \eta, \tau, \xi}=\sum_{k=0}^{\beta} {\beta \choose k} \left( \eta-\frac{\tau}{\xi} \right)^{k} \left( \frac{\tau}{\xi} \right)^{\beta-k} \Gamma(1-[\beta-k]\xi),
\Eeq
and the expression of $\mathrm{Cov} \left( Z(\Mb{0})^{\beta}, Z(\Mb{x})^{\beta} \right)$ is given by \eqref{Eq_Cov_Maxstab_Real_Marg_Eq_Coeff}.
\end{Th}
\begin{proof}
Such a Brown--Resnick random field is sample-continuous \citep[][proof of Theorem 3]{koch2017TCL} and thus measurable, which yields by Lemma \ref{Lemma_CP_Locally_Integrable_Sample_Paths} that $C \in \mathcal{C}$. Moreover, $C$ is stationary and therefore has a constant expectation. In addition, by Lemma \ref{Lem_Moment_2plusdelta}, $D_{\beta, \eta, \tau, \xi}$ satisfies the assumptions on the function $F$ of Theorem 3 in \cite{koch2017TCL}. Thus, the latter theorem yields that $C$ satisfies the CLT. Finally, \eqref{Ex_Exp_Z_beta} yields $\mathbb{E} \left[ C(\Mb{0}) \right]=\mu_{\beta, \eta, \tau, \xi}$. The result follows from Theorem 2 in \cite{koch2019SpatialRiskAxioms}. 
\end{proof}
}
If $\lambda$ is large enough, this result gives an approximation for the distribution of the normalized spatially aggregated loss:
$$ L_N(\lambda A, C) \approx \mathcal{N} \left( \mu_{\beta, \eta, \tau, \xi}, \frac{1}{\lambda^2 \nu(A)} \int_{\mathbb{R}^2} \mathrm{Cov} \left( Z(\Mb{0})^{\beta}, Z(\Mb{x})^{\beta} \right) \nu(\mathrm{d}\Mb{x}) \right),$$
where $\approx$ means ``approximately follows''. Such an approximation is useful in practice, e.g., for an insurance company. \Tb{For the same reasons as those mentioned in Section \ref{Subsubsec_Axioms}, these results generally hold true when the focus in on damage due to extreme wind speeds.}

\subsection{Spatial risk measures associated with \Tb{value-at-risk} and expected shortfall}

For a random variable $X$ with distribution function $F_{X}$, its \Tb{value-at-risk at confidence level} $\alpha \in (0,1)$ is written $\mbox{VaR}_{\alpha}(X)=\inf \left \{ x \in \mathbb{R}: F_{X}(x) \geq \alpha \right \}$.
Moreover, provided $\mathbb{E} [ | X | ]<\infty$, its expected shortfall at \Tb{confidence} level $\alpha \in(0,1)$ is defined by 
$$ \mathrm{ES}_{\alpha}(X)= \frac{1}{1-\alpha} \int_{\alpha}^{1} \mathrm{VaR}_{u}(X) \  \nu(\mathrm{d}u).$$ 
Classic values for $\alpha$ are $0.95$ and $0.99$. In the actuarial literature, ES is sometimes referred to as \Tb{tail value-at-risk} \citep[e.g.,][Definition 2.4.1]{denuit2005actuarial}. In the following, for $\alpha \in (0,1)$, $q_{\alpha}$ and $\phi$ denote the quantile at level $\alpha$ and the density of the standard Gaussian distribution, respectively. In this section, we focus on  
$$ \mathcal{R}_{3, \alpha}(\cdot, C) =\mbox{VaR}_{\alpha} ( L_N(\cdot, C) ), \quad  \mbox{and} \quad \mathcal{R}_{4, \alpha}(\cdot, C) =\mbox{ES}_{\alpha} (L_N(\cdot, C)),$$
where $\alpha \in (0,1)$ and $C$ is given in \eqref{Eq_Specific_Cost_Process}. We first shortly comment on the functions $\lambda \mapsto \mathcal{R}_{3, \alpha}(\lambda A, C)$ and $\lambda \mapsto \mathcal{R}_{4, \alpha}(\lambda A, C)$, for $A \in \mathcal{A}$, and then provide conditions on the field $Z$ such that $\mathcal{R}_{3, \alpha}(\cdot, C)$ and $\mathcal{R}_{4, \alpha}(\cdot, C)$ satisfy (at least partially) the axioms in \Tb{Section \ref{Subsec_Spat_Risk_Axioms}}.

\subsubsection{Study of $\mathcal{R}_{3, \alpha}(\lambda A, C)$ and $\mathcal{R}_{4, \alpha}(\lambda A, C)$, $\lambda >0$}

Deriving a tractable formula for VaR of $L_N(\lambda A, C)$, $A \in \mathcal{A}$, is very difficult. The same type of approximation as that described in \citet[][Section 4.3.1]{koch2017spatial} can be used, leading to \Tb{similar} graphs as Figure 6 in \cite{koch2017spatial} for the function $\lambda \mapsto \mathcal{R}_{3, \alpha}(\lambda A, C)$, but this approach is numerically rather time-consuming.

For the same reasons, it is arduous to obtain a tractable formula for ES of $L_N(\lambda A, C)$, $A \in \mathcal{A}$. \Tb{Nonetheless, as
\Beq
\label{Eq_Expression_TVaR_CTE}
\mathrm{ES}_{\alpha}(X)=\mathbb{E} \left [ X | X > \mathrm{VaR}_{\alpha}(X) \right ]
\Eeq
for a continuous random variable $X$, it is possible to approximate ES of $L_N(\lambda A, C)$ by estimating the right-hand side of \eqref{Eq_Expression_TVaR_CTE} with a Monte-Carlo method; this is however time-consuming.}

\subsubsection{Axioms}

\Tb{We do not have an explicit formula for $\mathcal{R}_{3, \alpha}(\lambda A, C)$ and $\mathcal{R}_{4, \alpha}(\lambda A, C)$, but outcomes from \cite{koch2019SpatialRiskAxioms} (connected with Theorem \ref{TCL_Wind}) yield} their asymptotic behaviours (when $\lambda \to \infty$) under some conditions. The corresponding result is part of next theorem.

\Tb{
\begin{Th}
\label{Th_Axioms_R3_R4}
\Ben
\item Let $\{Z(\Mb{x})\}_{\Mb{x} \in \mathbb{R}^2}$ be a stationary and measurable max-stable random field with GEV parameters $\eta \in \mathbb{R}$, $\tau>0$ and $\xi \in \mathbb{R}$. Let $\beta \in \mathbb{N}_*$ and $\{ C(\Mb{x}) \}_{\Mb{x} \in \mathbb{R}^2}=\{ Z(\Mb{x})^{\beta} \}_{\Mb{x} \in \mathbb{R}^2}$ such that $C$ has a.s. locally integrable sample paths (this is satisfied, e.g., if $\beta \xi <1$). Then, for all $\alpha \in (0,1)$, $\mathcal{R}_{3, \alpha}(\cdot, C)$ satisfies the axiom of spatial invariance under translation. Provided $\mathbb{E}\left[ \left | C(\Mb{0}) \right | \right] < \infty$, the same is true for $\mathcal{R}_{4, \alpha}(\cdot, C)$. These results hold true, e.g., for the tube random field and the measurable Schlather and Brown--Resnick fields (including the Smith field) if $\beta \xi <1$.
\item Let $Z$, $\beta$ and $C$ be as in Theorem \ref{Th_Axioms_R2}, Point 3. Then:
\Ben
\item For all $\alpha \in (0,1) \backslash \{ 1/2 \}$, $\mathcal{R}_{3,\alpha}(\cdot, C)$ satisfies the axiom of asymptotic spatial homogeneity of order $-1$ with 
$$ K_1(A, C)=\mu_{\beta, \eta, \tau, \xi} \quad
\mbox{and} \quad 
K_2(A,C)= \frac{ q_{\alpha}}{\sqrt{\nu(A)}} \sqrt{ \int_{\mathbb{R}^2} \mathrm{Cov} \left( Z(\Mb{0})^{\beta}, Z(\Mb{x})^{\beta} \right) \nu(\mathrm{d}\Mb{x}) },\quad A \in \mathcal{A}_c.$$
\item For all $\alpha \in (0,1)$, $\mathcal{R}_{4,\alpha}(\cdot, C)$ satisfies the axiom of asymptotic spatial homogeneity of order $-1$ with 
$$ K_1(A, C)=\mu_{\beta, \eta, \tau, \xi} \quad
\mbox{and} \quad 
K_2(A, C)=\frac{\phi(q_{\alpha})}{\sqrt{\nu(A)} (1-\alpha) } \sqrt{ \int_{\mathbb{R}^2} \mathrm{Cov} \left( Z(\Mb{0})^{\beta}, Z(\Mb{x})^{\beta} \right) \nu(\mathrm{d}\Mb{x})} , \quad A \in \mathcal{A}_c.$$
\Een
The expressions of $\mu_{\beta, \eta, \tau, \xi}$ and $\mathrm{Cov} \left( Z(\Mb{0})^{\beta}, Z(\Mb{x})^{\beta} \right)$ are given in \eqref{Eq_mu_beta_eta_tau_xi} and \eqref{Eq_Cov_Maxstab_Real_Marg_Eq_Coeff}, respectively.
\Een
\end{Th}
}
\begin{proof} 
\Tb{
1. The field $C$ belongs to $\mathcal{C}$ by assumption and is stationary by stationarity of $Z$. The fact that $C$ has a.s. locally integrable sample paths when $\beta \xi <1$ directly follows from Lemma \ref{Lemma_CP_Locally_Integrable_Sample_Paths}. Now, $\mathcal{R}_{3, \alpha}(\cdot, C)$ is well-defined. In addition, we have, for all $\lambda>0$, that
$$ | L_N(\lambda A, C) | \leq \frac{1}{\nu(\lambda A)} \int_{\lambda A} | C(\Mb{x}) | \ \nu(\mathrm{d} \Mb{x}).$$
Accordingly, if $\mathbb{E}\left[ \left | C(\Mb{0}) \right | \right] < \infty$, the stationarity of $C$ and Fubini's theorem entail that $\mathbb{E} \left[ | L_N(\lambda A, C) | \right]< \infty$, which implies that $\mathcal{R}_{4, \alpha}(\cdot, C)$ is well-defined. Finally, VaR and ES are both law-invariant classical risk measures. Consequently, Theorem 5, Point 1, in \cite{koch2019SpatialRiskAxioms} gives the first result. The end of the proof of Theorem \ref{Th_Axioms_R2}, Point 1, and the fact that $\beta \xi<1$ implies $\mathbb{E}\left[ \left | C(\Mb{0}) \right | \right] < \infty$ show that the specific fields mentioned satisfy the required assumptions, concluding the proof.}

\Tb{
2. The arguments are the same as in the proof of Theorem \ref{Th_Axioms_R2}, Point 3. 
}
\end{proof}
\Tb{As above, these results generally hold true in the case of the risk of extreme wind speeds.} 

\begin{Rq}
\label{Rq_Ext_Complex}
\Tb{The results of Theorem \ref{Th_Axioms_R2}, Point 3, Theorem \ref{TCL_Wind} and Theorem  \ref{Th_Axioms_R3_R4}, Point 2, hold true for a more general Brown--Resnick field, more specifically satisfying the conditions of Remark 3 in \cite{koch2017TCL} or, equivalently, Theorem 9 in \cite{koch2019SpatialRiskAxioms}.}
\end{Rq}

\section{Conclusion}
\label{Sec_Conclusion}

\Tb{
As detailed in the paper, literature suggests that power-laws are appropriate wind damage functions but that the suitable power varies significantly from a situation to the other, most often from $2$ to $12$; especially it is an increasing function of the deductible of the insurance contract. Thus, the study of powers of extremal fields is insightful for the assessment of the risk of losses due to extreme wind speeds, and analyzing the sensitivity to the value of the power is needed. In this paper, we thoroughly investigate the correlation structure of powers of the Brown--Resnick max-stable random field, which is a very suited model for spatial extremes. Even if our primary focus is on risk assessment of damaging wind speeds, the obtained results may be valuable for the extreme-value community regardless of any notion of risk. Then, we illustrate the concepts of spatial risk measure and corresponding axioms introduced in \cite{koch2017spatial, koch2019SpatialRiskAxioms} in a context where the cost fields are precisely powers of max-stable fields. Using the previous part, we perform a comprehensive study of spatial risk measures associated with variance and induced by such cost fields. In addition, we show that under relatively mild conditions which are typically satisfied for the risk of damaging extreme wind speeds, the induced spatial risk measures associated with several classical risk measures satisfy (at least part of) the axioms. Variance, VaR and ES lead to asymptotic spatial homogeneity of order $-2$, $-1$ and $-1$, respectively. Our results are valuable for risk assessment in actuarial practice of, e.g., extratropical and tropical cyclones.}

\Tb{
Ongoing work consists, inter alia, in applying the results of the current paper to the pricing of concrete reinsurance treaties as well as event-linked securities such as catastrophe bonds.
}

\section*{Acknowledgements}
The author wishes to thank Anthony C. Davison, Paul Embrechts, Pierre Ribereau and Christian Y. Robert for some comments. He also would like to acknowledge the MIRACCLE-GICC project, RiskLab at ETH Zurich, the Swiss Finance Institute, the Swiss National Science Foundation (project 200021\_178824) and the Institute of Mathematics at EPFL for financial support.

\newpage

\appendix

\section{\Tb{Case of simple max-stable random fields and $\beta<1/2$ or $\beta<1$}}
\label{Sec_Appendix_Maxstable}

\Tb{This appendix explains that the results obtained above basically apply if the max-stable random field $Z$ is simple (instead of having general GEV margins) and $\beta<1/2$ or $\beta<1$. The point is that $\beta$ is allowed to take any value below these upper bounds and does not need to be an integer anymore. As standard Fr\'echet margins are far from being realistic for wind speed extremes, the interest of this section mostly lies in a better understanding of some properties of max-stable fields.}

\Tb{First we consider the dependence measure $\mathrm{Corr} ( Z^{(\mathrm{s})}(\Mb{x}_1)^{\beta}, Z^{(\mathrm{s})}(\Mb{x}_2)^{\beta})$, where $\{ Z^{(\mathrm{s})}(\Mb{x}) \}_{\Mb{x} \in \mathbb{R}^2}$ is a simple Brown--Resnick max-stable random field and $\beta < 1/2$, i.e., $\mathcal{D}_{\beta, 1, 1, 1}(\Mb{x}_1, \Mb{x}_2)$. The condition $\beta \xi < 1/2$ with $\beta \in \mathbb{N}_*$ of \eqref{Eq_DepMeas} translates into $\beta < 1/2$; any negative value is allowed as simple max-stable fields are a.s. positive. We introduce, for $\beta<1/2$,  
$$
g_{\beta}^{(\mathrm{s})}(h) =
\left \{
\begin{array}{ll}
\Gamma(1-2 \beta) & \mbox{if} \quad  h=0, \\ 
\displaystyle \int_{0}^{\infty} \theta^{\beta} \Big[ C_2(\theta,h) \  C_1(\theta,h)^{2 \beta -2} \ \Gamma(2-2 \beta) 
+ C_3(\theta,h) \ C_1(\theta,h)^{2 \beta -1} \  \Gamma(1-2 \beta) \Big] \  \nu(\mathrm{d}\theta) & \mbox{if} \quad h>0,
\end{array}
\right.
$$
which arises when setting $\beta_1=\beta_2$ in the function $g_{\beta_1, \beta_2}^{(\mathrm{s})}$ specified in \eqref{Eq_Def_g_beta1_beta2}.
Denoting by $\gamma_W$ the variogram of $Z^{(\mathrm{s})}$, it follows from Theorem \ref{Th_Expression_Covariance_Damages} that, for all $\Mb{x}_1, \Mb{x}_2 \in \mathbb{R}^2$ and $\beta < 1/2$,
$\mathrm{Cov} ( Z^{(\mathrm{s})}(\Mb{x}_1)^{\beta}, Z^{(\mathrm{s})}(\Mb{x}_2)^{\beta} )=g_{\beta}^{(\mathrm{s})}(\sqrt{\gamma_W(\Mb{x}_2-\Mb{x}_1)}) - \left[ \Gamma(1- \beta) \right]^2$. Our dependence measure (provided that $\beta \neq 0$) is readily derived and its behaviour is similar to the one we observed in Section \ref{Sec_Dependence_Power_Maxstable} (not shown).} 

\Tb{We now investigate the function $g_{\beta}^{(\mathrm{s})}$ in further details.
Very similar proofs as for Propositions~\ref{Prop_Decrease_gtilde}--\ref{Prop_Lim_gtilde_infty} yield, for $\beta, \beta_1, \beta_2 <1/2$, that the functions $g_{\beta_1, \beta_2}^{(\mathrm{s})}$ defined in \eqref{Eq_Def_g_beta1_beta2} and $g_{\beta}^{(\mathrm{s})}$ are strictly decreasing, $\lim_{h \to 0} g_{\beta}^{(\mathrm{s})}(h)=\Gamma(1-2 \beta)$
(implying that $g_{\beta}^{(\mathrm{s})}$ is continuous everywhere on $[0, \infty)$) and $\lim_{h \to \infty} g_{\beta}^{(\mathrm{s})}(h) = [\Gamma(1-\beta)]^2$. This entails that, for any $h \geq 0$, $\lim_{\beta \to -\infty} g_{\beta}^{(\mathrm{s})}(h)= \infty$.
Figure \ref{Plot_Persp_gbeta}, obtained using adaptive quadrature with a relative accuracy of $10^{-5}$, shows that the decrease of $g_{\beta}^{(\mathrm{s})}(h)$ for a given $\beta$ with respect to $h$ is more and more pronounced when $|\beta|$ increases, and that, for $h$ fixed, the absolute value of the slope of $g_{\beta}^{(\mathrm{s})}(h)$ increases very fast with $|\beta|$, in link with rapid divergence to $\infty$. Obviously, the behaviour of $\mathrm{Cov} ( Z^{(\mathrm{s})}(\Mb{x}_1)^{\beta}, Z^{(\mathrm{s})}(\Mb{x}_2)^{\beta})$ is similar; the same holds true for $g_{\beta_1, \beta_2}^{(\mathrm{s})}$.
\begin{figure}[h!]
\center
\includegraphics[scale=0.8]{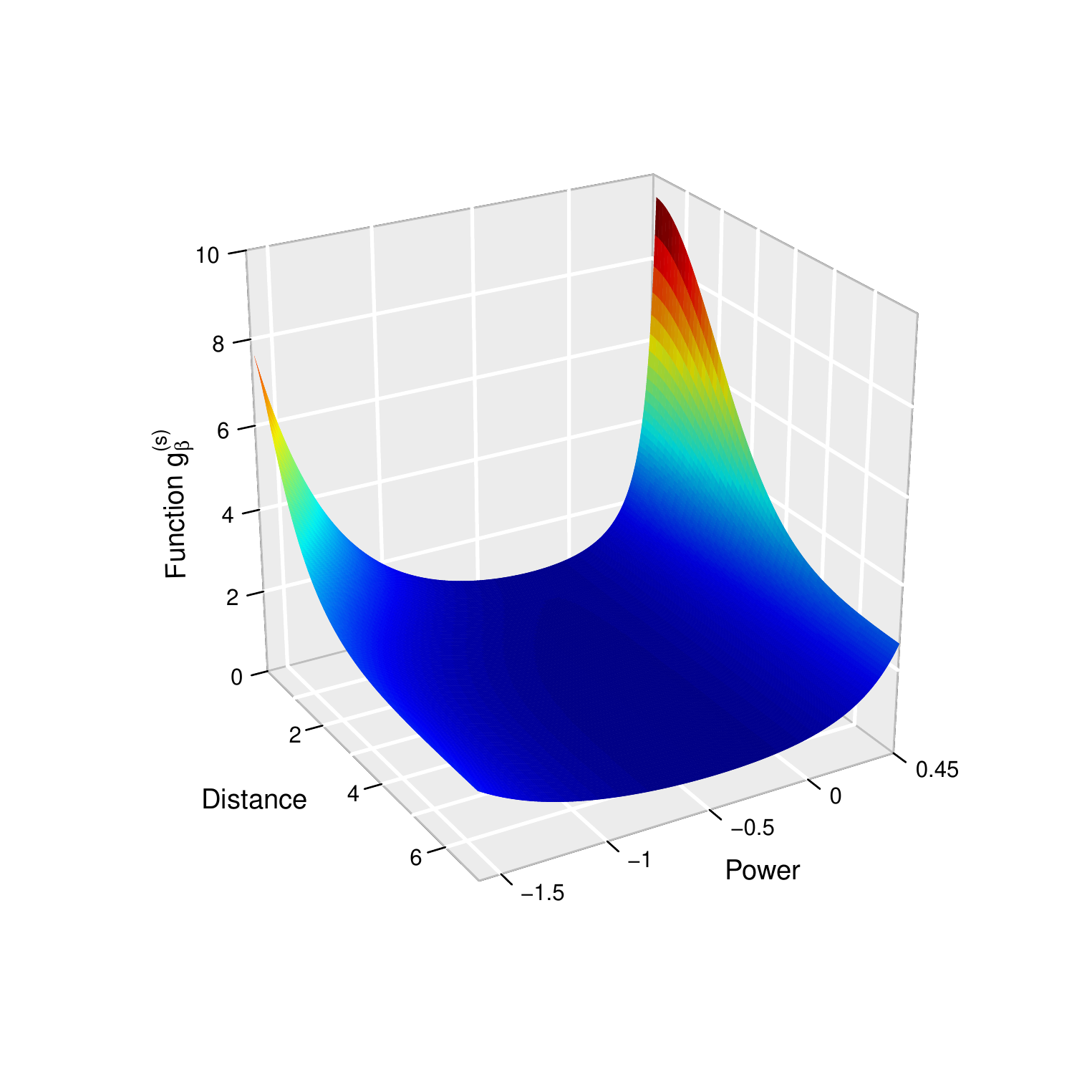}
\caption{Evolution of the function $g_{\beta}^{(\mathrm{s})}$ with respect to the distance $h$ and the power $\beta$ for $\beta \in [-1.6, 0.45]$.}
\label{Plot_Persp_gbeta}
\end{figure}
}

\Tb{The next theorem gives the expression of $\mathcal{R}_2(\lambda A, C)$ when $Z^{(\mathrm{s})}$ is a simple Brown--Resnick random field with an isotropic variogram, $A$ is either a disk or a square, and $\beta < 1/2$.
\begin{Th}
\label{Th_Behaviour_R2_Simple_Max}
Let $\{ Z^{(\mathrm{s})}(\Mb{x}) \}_{\Mb{x} \in \mathbb{R}^2}$ be a measurable simple Brown--Resnick random field associated with an isotropic variogram $\gamma_W$ whose corresponding univariate function is $\gamma_{W, \mathrm{u}}$. 
Let $\beta < 1/2$ and $\{ C(\Mb{x}) \}_{\Mb{x} \in \mathbb{R}^2}=\{ Z^{(\mathrm{s})}(\Mb{x})^{\beta} \}_{\Mb{x} \in \mathbb{R}^2}$. Finally, let $f_d$ and $f_s$ be as in Theorem \ref{Th_R2_Lambda_Finite_BR}. Then:
\Ben
\item Let $A$ be a disk with radius $R$. For all $\lambda >0$, we have
$$
\mathcal{R}_2(\lambda A, C)=- \left[ \Gamma(1-\beta) \right]^2 + \int_{h=0}^{2R} f_d(h,R) \ g_{\beta}^{(\mathrm{s})} \left( \sqrt{\gamma_{W, \mathrm{u}}(\lambda h)} \right)\  \nu(\mathrm{d}h).
$$
\item Let $A$ be a square with side $R$. For all $\lambda >0$, we have
$$\mathcal{R}_2(\lambda A, C)= - \left[ \Gamma(1-\beta) \right]^2 + \int_{h=0}^{R \sqrt{2}} f_s(h,R) \ g_{\beta}^{(\mathrm{s})} \left( \sqrt{\gamma_{W, \mathrm{u}}(\lambda h)} \right)\  \nu(\mathrm{d}h).$$
\Een
\end{Th}
We also have the equivalent of Theorem \ref{Th_Limiting_Risk_Measure}.
\begin{Th}
\label{Th_Limiting_Risk_Measure_Simple_Max}
Let $Z^{\mathrm{(s)}}$, $\beta$ and $C$ be as in Theorem \ref{Th_Behaviour_R2_Simple_Max}. Furthermore, assume that $\gamma_{W, \mathrm{u}}$ is measurable and satisfies
$$
\lim_{h \to \infty} \gamma_{W, \mathrm{u}}(h)= \infty.
$$
Then, for all $A$ being a disk with radius $R$ or a square with side $R$, $\lim_{\lambda \to \infty} \mathcal{R}_2 (\lambda A, C)=0$.
\end{Th}
Theorem \ref{Th_Limiting_Risk_Measure_Simple_Max} is consistent with the current knowledge about mixing of max-stable random fields. Under the conditions of Theorem \ref{Th_Limiting_Risk_Measure_Simple_Max}, the extremal coefficient function $\Theta$ of the Brown--Resnick field is isotropic and so we introduce the function $\Theta_{\mathrm{u}}: [0, \infty) \to [0, 2]$ such that, for all $\Mb{x}_1, \Mb{x}_2 \in \mathbb{R}^2$, $\Theta(\Mb{x}_1, \Mb{x}_2)=\Theta_{\mathrm{u}}(\| \Mb{x}_2 - \Mb{x}_1 \|)$. As $\Theta_{\mathrm{u}}(h)=2 \Phi(\sqrt{\gamma_{W, \mathrm{u}}(h)}/2)$ \citep[e.g.,][]{davison2012statistical}, we have $\lim_{h \to \infty} \Theta_{\mathrm{u}}(h)=2$, which implies by the fact that Theorem 3.1 in \cite{kabluchko2010ergodic} can be extended to random fields on $\mathbb{R}^d$, $d>1$ \citep[e.g.,][p.20]{DombryHDR2012}, that the Brown--Resnick field is mixing under the assumptions of Theorem \ref{Th_Limiting_Risk_Measure}. Thus, it is mean-ergodic, which entails the result of Theorem \ref{Th_Limiting_Risk_Measure}.
}

\Tb{
Concerning the axioms, Theorem \ref{Th_Axioms_R2}, Theorem \ref{TCL_Wind}, Remark \ref{Rq_Extens_BR}, Point 2 of Theorem \ref{Th_Axioms_R3_R4} and Remark \ref{Rq_Ext_Complex} also hold true if the field $Z$ satisfies the same assumptions but is simple (instead of having general GEV margins) and $\beta < 1/2$ (not required to be a positive integer). Similarly, Point 1 of Theorem \ref{Th_Axioms_R3_R4} is also true if $Z$ satisfies the same conditions but is simple and $\beta <1$.
}

\newpage

\bibliographystyle{apalike}
\bibliography{References_Erwan}
\end{document}